\documentclass[letterpaper,reqno,12pt]{amsart}
\usepackage[margin=1.2in]{geometry}

\usepackage{amssymb, enumerate}
\usepackage{mathrsfs}
\usepackage[all]{xy}
\usepackage[dvipdfmx]{hyperref}
\hypersetup{colorlinks=true}

\theoremstyle{plain}
\newtheorem{thm}{Theorem}[section] 
\newtheorem{cor}[thm]{Corollary}
\newtheorem{prop}[thm]{Proposition}

\newtheorem{lem}[thm]{Lemma}

\theoremstyle{definition} 
\newtheorem{defn}[thm]{Definition}

\newtheorem*{notation}{Notation}

\theoremstyle{remark}
\newtheorem{rem}[thm]{Remark}
\newtheorem{ques}[thm]{Question}
\newtheorem*{cl}{Claim}
\newtheorem*{acknowledgement}{Acknowledgments}

\def\ge{\geqslant}
\def\le{\leqslant}
\def\phi{\varphi}
\def\epsilon{\varepsilon}
\def\tilde{\widetilde}

\def\to{\longrightarrow}
\def\mapsto{\longmapsto}

\newcommand{\sO}{\mathcal{O}}

\newcommand{\Q}{\mathbb{Q}} 
\newcommand{\C}{\mathbb{C}}

\newcommand{\PP}{\mathbb{P}}

\newcommand{\m}{\mathfrak{m}}

\newsavebox{\circlebox}
\savebox{\circlebox}{\fontencoding{OMS}\selectfont\Large\char13}
\newlength{\circleboxwdht}

\def\Hom{\operatorname{Hom}}
\def\Spec{\operatorname{Spec}}

\def\depth{\operatorname{depth}}

\title{Weak Akizuki-Nakano vanishing theorem\\ for globally $F$-split 3-folds}

\author{Kenta Sato}
\address{Faculty of Mathematics, Kyushu University, 744 Motooka, Nishi-ku, Fukuoka 819-
0395, Japan}
\email{ksato@math.kyushu-u.ac.jp}

\author{Shunsuke Takagi}
\address{Graduate School of Mathematical Sciences, University of Tokyo, 3-8-1 Komaba, Meguro-ku, Tokyo 153-8914, Japan}
\email{stakagi@ms.u-tokyo.ac.jp}

\thanks{}


\dedicatory{Dedicated to Professor~Bernd~Ulrich on the occasion of his sixty-fifth birthday.}

\begin{document}

\begin{abstract}
In this paper, we prove that a weak form of the Akizuki-Nakano vanishing theorem holds on globally $F$-split 3-folds. 
Making use of this vanishing theorem, we study deformations of globally $F$-split Fano 3-folds and the Kodaira vanishing theorem for thickenings of locally complete intersection globally $F$-regular 3-folds. 
\end{abstract}

\maketitle

\markboth{K.~SATO and S.~TAKAGI}{AKIZUKI-NAKANO VANISHING FOR GLOBALLY $F$-SPLIT $3$-FOLDS}

\section{Introduction}
The Akizuki-Nakano vanishing theorem is a generalization of the Kodaira vanishing theorem and plays  an important role in complex algebraic geometry. 
It was generalized to singular varieties in the following form. 
\begin{thm}\label{AN char.0 intro}
Let $X$ be a $d$-dimensional complex projective variety and $\mathcal{L}$ be an ample line bundle on $X$. 
\begin{enumerate}
\item[\textup{(1)}] \textup{(\cite{GKP}, \cite{KSS})} If $X$ is Cohen-Macaulay log canonical of dimension $d \le 3$, then 
for every nonnegative integers $i, j$ with $i+j<d$, one has 
\[
H^i(X, \Omega_X^{[j]} \otimes \mathcal{L}^{-1})=0, 
\]
where $\Omega_X^{[j]}$ is the reflexive hull of $\Omega_X^j=\bigwedge^j \Omega_{X/\C}$. 
\item[\textup{(2)}] \textup{(\cite{Hamm})} Let $s$ denote the dimension of the singular locus of $X$. 
If $X$ is a locally complete intersection, then 
for every nonnegative integers $i, j$ with $i+j<d-s$, one has 
\[
H^i(X, \Omega_X^j \otimes \mathcal{L}^{-1})=0.  
\]
\end{enumerate}
\end{thm}
In this paper, we pursue a characteristic $p$ analog of Theorem \ref{AN char.0 intro}. 
Let $X$ be a projective variety over a perfect field $k$ of characteristic $p>0$. 
Since the Kodaira vanishing theorem fails in positive characteristic in general, we need to impose an additional condition on $X$. 
We assume that $X$ is Cohen-Macaulay and globally $F$-split, that is, the Frobenius map $\sO_X \to F_*\sO_X$ splits as an $\sO_X$-module homomorphism. Then $X$ satisfies the Kodaira vanishing theorem. 
Moreover, if $X$ is smooth, then it lifts to the ring $W_2(k)$ of Witt vectors of length 2 (see \cite[p.164]{Il2} or \cite[Corollary 9.2]{Ki}) and consequently satisfies the Akizuki-Nakano vanishing theorem when $p \ge \dim X$ by a result of Raynaud \cite[Corollaire 2.8]{DI}. 

On the other hand, very little is known when $X$ is singular. 
Kawakami \cite{Kaw} recently proved that the Akizuki-Nakano vanishing theorem holds on globally $F$-regular surfaces, using Graf's extension theorem \cite{Gr}. 
Here, globally $F$-regular varieties are a class of globally $F$-split varieties containing toric and Schubert varieties. 
Several Kodaira-type vanishing theorem are known to hold on globally $F$-regular varieties (see \cite{GT}, \cite{SS}).
In this paper, we investigate the Akizuki-Nakano vanishing theorem for globally $F$-split or globally $F$-regular 3-folds. 
Our main result is stated as follows. 

\begin{thm}[Theorem \ref{AN vanishing}, Remark \ref{GFR case1}, Corollary \ref{GFR case2}]\label{main vanishing}
Let $X$ be a projective globally $F$-regular variety over a perfect field of characteristic $p>0$ and $\mathcal{L}$ be a globally generated ample line bundle on $X$. 
Assume in addition that $X$ is of dimension $d \le 3$ and the morphism $\Phi_{\mathcal{L}}:X \to \PP^N=\PP(H^0(X, \mathcal{L})^*)$ induced by $\mathcal{L}$ is generically \'etale on its image.\footnote{The latter condition is satisfied for example if $\mathcal{L}$ is very ample.} 
\begin{enumerate}
\item[\textup{(1)}] If $X$ has only isolated singularities, then for every nonnegative integers $i, j$ with $i+j<d$, one has 
\[
H^i(X, \Omega_X^{[j]} \otimes \mathcal{L}^{-1})=0. 
\]
\item[\textup{(2)}] Let $s$ denote the dimension of the singular locus of $X$. 
If $X$ is a locally complete intersection and $p \ge 5$, then for every nonnegative integers $i, j$ with $i+j<d-s$, one has 
\[
H^i(X, \Omega_X^j \otimes \mathcal{L}^{-1})=0.  
\]
\end{enumerate}
\end{thm}

Making use of (a variant of) Theorem \ref{main vanishing}, we study properties of globally $F$-split Fano 3-folds. 
In particular, we focus on deformations of globally $F$-split Fano 3-folds with only isolated complete intersection singularities. 
The following theorem should be compared with a result of Namikawa \cite{Na} about deformations of complex Fano 3-folds with only Gorenstein terminal singularities. 
\begin{thm}[Theorem \ref{namikawa result1}, Theorem \ref{namikawa result2}]\label{deform Intro}
Let $X$ be a projective globally $F$-split Fano 3-fold over an algebraically closed field of characteristic $p>0$ which has only isolated complete intersection singularities.  
Suppose that there exists an integer $1 \le i_0 \le p$ such that $|-i_0K_X|$ is base point free and the morphism $\varphi_{|-i_0K_X|}:X \to \mathbb{P}^N=\mathbb{P}(H^0(X,\sO_X(-i_0K_X))^*)$ induced by $|-i_0K_X|$ is generically \'etale on its image.\footnote{This condition is satisfied for example if $|-K_X|$ is base point free and $p \ge 5$.} 
\begin{enumerate}
\item[\textup{(1)}] The deformations of $X$ are unobstructed and, in particular, $X$ is liftable to characteristic zero. 
\item[\textup{(2)}] If $X$ has only ordinary double points, then $H^2(X,T_X)=0$ and, in particular, $X$ admits a smoothing. 
\end{enumerate}
\end{thm}

Since the modulo $p$ reduction of a complex Fano 3-fold with only Gorenstein terminal singularities is a globally $F$-split 3-fold with only isolated complete intersection singularities for sufficiently large primes $p$ (see \cite[Theorem 1.2]{SS}), Theorem \ref{deform Intro} gives a purely algebraic proof of the result of Namikawa. 

As another application of Theorem \ref{main vanishing}, we study the Kodaira vanishing theorem for thickenings. 
Bhatt-Blickle-Lyubeznik-Singh-Zhang \cite{BBLSZ} proved that the Kodaira vanishing theorem holds on thickenings of a locally complete intersection closed subvariety of the projective space $\PP^N_k$ over a field $k$ of characteristic zero. 
The following is a characteristic $p$ analog of their result in dimension 3. 
\begin{thm}[Theorem \ref{thicknings in char. p}]
Let $X$ be a locally complete intersection closed subvariety of the projective space $\PP^N_k$ over a perfect field $k$ of characteristic $p \ge 5$. 
For each integer $t \ge 1$, let $X_t \subseteq \PP^N_k$ be the $t$-th thickening of $X$, that is, the closed subscheme defined by $\mathcal{I}^t$, where $\mathcal{I} \subseteq \sO_{\PP^N_k}$ is the defining ideal sheaf of $X$ in $\PP^N_k$. 
Suppose that $X$ is globally $F$-regular of dimension $d \le 3$. 
Assume in addition that one of the following conditions is satisfied: 
\begin{enumerate}
\item [$\textup{(i)}$] $t$ is less than or equal to $p$, 
\item [$\textup{(ii)}$] the normal bundle $\mathcal{N}_{X/\PP^N}$ is a direct sum of line bundles.
\end{enumerate}
Then for every integers $i < d-s$ and $\ell \ge 1$, we have 
\[H^i(X_t, \sO_{X_t}(-\ell))=0,\] 
where $s$ is the dimension of the singular locus of $X$.
\end{thm}
We remark that our argument gives an alternative proof of the result of Bhatt-Blickle-Lyubeznik-Singh-Zhang in characteristic zero (see Remark \ref{alternative proof}). 

\begin{acknowledgement}
The authors are grateful to Kenta Hashizume, Tatsuro Kawakami, Yujiro Kawamata, Yoichi Miyaoka, Shigefumi Mori, Noboru Nakayama, Taro Sano, Vasudevan Srinivas and Bernd Ulrich for valuable comments and discussions. 
They also thank the anonymous referee for helpful comments.
The first author was supported by RIKEN iTHEMS Program.
The second author was partially supported by JSPS KAKENHI Grant Numbers JP15H03611, JP16H02141, and JP17H02831.
\end{acknowledgement}

\begin{notation}
Throughout this paper, all schemes are assumed to be Noetherian and separated. 
A variety is an integral scheme of finite type over a field and a 3-fold (resp. a curve, a surface) is a 3-dimensional (resp. 1-dimensional, 2-dimensional) variety. 
Given a variety $X$ over a field $k$, $T_X$ denotes the tangent sheaf of $X$, that is, the dual $\mathcal{H}\mathrm{om}(\Omega_X, \sO_X)$ of the cotangent sheaf $\Omega_X=\Omega_{X/k}$.
\end{notation}

\section{Preliminaries}
\subsection{Globally \texorpdfstring{$F$-split}{F-split} and globally \texorpdfstring{$F$-regular}{F-regular} varieties}
In this subsection, we recall the definitions and basic properties of globally $F$-split and globally $F$-regular varieties. 

\begin{defn}[\cite{BK}, \cite{HWY}, \cite{MR}, \cite{SS}, \cite{Sm2}]
Let $X$ be a variety over a perfect field of characteristic $p>0$.
\begin{enumerate}
\item[(i)] $X$ is said to be \textit{globally $F$-split}\footnote{Globally $F$-split varieties are often called simply $F$-split. However, we do not use this terminology in this paper, because it may be confused with locally $F$-split varieties.} if the Frobenius map $\sO_X \to F_*\sO_X$ splits as an $\sO_X$-module homomorphism. 
\end{enumerate}
We assume in addition that $X$ is normal, and let $\Delta=\sum_i d_i \Delta_i$ be an effective $\Q$-Weil divisor on $X$. 
For each integer $e \ge 1$, the round-up of $(p^e-1)\Delta$ is $\lceil (p^e-1)\Delta \rceil=\sum_i \lceil (p^e-1)d_i \rceil \Delta_i$, where $\lceil (p^e-1)d_i \rceil$ denotes the smallest integer greater than or equal to $(p^e-1)d_i$. 
Also, given an integer $e \ge 1$ and an effective Weil divisor $B$ on $X$, let $\varphi^{(e)}_B:\sO_X \to F^e_*\sO_X(B)$ denote the composite map 
\[
\sO_X \to F^e_*\sO_X \hookrightarrow F^e_*\sO_X(B),
\]
where $\sO_X \to F^e_*\sO_X$ is the $e$-times iterated Frobenius map and $F^e_*\sO_X \hookrightarrow F^e_*\sO_X(B)$ is the push-forward of the natural inclusion $\sO_X \hookrightarrow \sO_X(B)$ by the $e$-times iterated Frobenius morphism $F^e:X \to X$. 
\begin{enumerate}
\item[(ii)] The pair $(X, \Delta)$ is said to be \textit{globally sharply $F$-split} if there exists an integer $e \ge 1$ such that $\varphi^{(e)}_{\lceil (p^e-1)\Delta \rceil}:\sO_X \to F^e_*\sO_X(\lceil (p^e-1) \Delta \rceil)$
splits as an $\sO_X$-module homomorphism. 
$X$ is globally $F$-split if and only if the pair $(X,0)$ is globally sharply $F$-split. 
\item[(iii)] The pair $(X, \Delta)$ is said to be \textit{globally $F$-regular} if for every effective Weil divisor $D$ on $X$, there exists an integer $e \ge 1$ such that $\varphi^{(e)}_{\lceil (p^e-1)\Delta \rceil+D}:\sO_X \to F^e_*\sO_X(\lceil (p^e-1) \Delta \rceil+D)$
splits as an $\sO_X$-module homomorphism. 
We simply say that $X$ is \textit{globally $F$-regular} if so is the pair $(X, 0)$.
\end{enumerate}
\end{defn}

The Kodaira vanishing theorem fails in positive characteristic in general, but it holds on globally $F$-split varieties. 
\begin{prop}\label{kodaira}
Let $X$ be an $n$-dimensional globally $F$-split projective variety over a perfect field of characteristic $p>0$ and $\mathcal{L}$ be an ample line bundle on $X$. 
\begin{enumerate}
\item[$(1)$] $($\cite[Theorem 1.2.8]{BK}$)$ $H^i(X, \mathcal{L})=0$ for all $i \ge 1$. 
\item[$(2)$] $($\cite[Theorem 1.2.9]{BK}$)$ If $X$ is Cohen-Macaulay, then $H^i(X, \mathcal{L}^{-1})=0$ for all $i<n$. 
\end{enumerate}
\end{prop}

\begin{lem}[\textup{\cite[Lemma 2.14]{GOST}}]\label{small}
Let $X$ be a normal quasi-projective variety over a perfect field of characteristic $p>0$. 
Let $f:X \dashrightarrow Y$ be a small birational map or an algebraic fiber space to a normal variety $Y$. 
If $X$ is globally $F$-split $($resp. globally $F$-regular$)$, then so is $Y$.
\end{lem}

\begin{lem}\label{crepant}
Let $X$ be a normal quasi-projective variety over a perfect field of characteristic $p>0$ and $\Delta$ be an effective $\Q$-Weil divisor on $X$ such that $K_X+\Delta$ is $\Q$-Cartier. 
Let $f:Y \to X$ be a proper birational morphism from a normal variety $Y$, 
and suppose that $\Delta_Y : = f^*(K_X+\Delta)- K_Y$ is an effective $\Q$-Weil divisor on $Y$. 
Then the following hold.
\begin{enumerate}
\item[$(1)$] $(Y, \Delta_Y)$ is globally $F$-regular if and only if so is $(X,\Delta)$.
\item[$(2)$] $(Y, \Delta_Y)$ is globally sharply $F$-split if and only if so is $(X,\Delta)$.
\end{enumerate}
\end{lem}

\begin{proof}
The assertion follows from essentially the same argument as the proof of \cite[Proposition 2.11]{HX}.
\end{proof}

\subsection{Singularities}
Here, we collect the definitions of singularities treated in this paper. 
Throughout this subsection, let $X$ be a normal variety over a perfect field $k$ and $\Delta$ be an effective $\Q$-Weil divisor on $X$. 
\begin{defn}
Suppose that $K_X+\Delta$ is $\Q$-Cartier. 
Given a projective birational morphism $\pi:Y \to X$ from a normal variety $Y$, we write 
\[
K_Y+\pi^{-1}_*\Delta=\pi^*(K_X+\Delta)+\sum_i a_i E_i, 
\]
where $\pi^{-1}_*\Delta$ is the strict transform of $\Delta$ by $\pi$, the $a_i$ are rational numbers and the $E_i$ are $\pi$-exceptional prime divisors on $Y$. 
We say that $(X, \Delta)$ is \textit{terminal} (resp.~ \textit{klt}) if all the $a_i >0$ (resp.~all the $a_i >-1$) for every projective birational morphism $\pi:Y \to X$ from a normal variety $Y$. 
\end{defn}

\begin{defn}
Suppose that $k$ is of characteristic $p>0$. 
The pair $(X, \Delta)$ is said to be \textit{strongly $F$-regular} if there exists an open covering $\{ U_i\}_{i \in I}$ of $X$ such that $(U_i, \Delta|_{U_i})$ is globally $F$-regular for every $i \in I$.
We simply say that $X$ is \textit{strongly $F$-regular} if so is the pair $(X,0)$.
\end{defn}

\begin{defn}[\cite{LT}]
We say that $X$ has only \textit{pseudo-rational} singularities if $X$ is Cohen-Macaulay and if the natural map $f_*\omega_Y \to \omega_X$ is an isomorphism for every projective birational morphism $f:Y \to X$ from a normal variety $Y$. 
\end{defn}

\begin{rem}\label{hierarchy}
Suppose that $k$ is of characteristic $p>0$ and $K_X+\Delta$ is $\Q$-Cartier. 
\begin{enumerate}
\item[(1)] (\cite[Theorem 2.29]{Ko}) If $(X, \Delta)$ is terminal, then $X$ is regular in codimension $2$.
\item[(2)] (\cite{HW}) If $(X, \Delta)$ is strongly $F$-regular, then it is klt.
\item[(3)] (\cite{Sm1}) If $X$ is strongly $F$-regular, then it is pseudo-rational.
\end{enumerate}
\end{rem}

\begin{rem}\label{GFR are CM}
Globally $F$-regular varieties are by definition strongly $F$-regular and, in particular, Cohen-Macaulay by Remark \ref{hierarchy} (3). 
\end{rem}

\begin{lem}\label{crepant SFR}
Suppose that $X$ is a normal quasi-projective variety over a perfect field of characteristic $p>0$ and $\Delta$ is an effective $\Q$-Weil divisor on $X$ such that $K_X+\Delta$ is $\Q$-Cartier.
Let $f:Y \to X$ be a proper birational morphism from a normal variety $Y$, and assume that $\Delta_Y : = f^*(K_X+\Delta)- K_Y$ is an effective $\Q$-Weil divisor on $Y$ and $(X, \Delta)$ is strongly $F$-regular.
Then $(Y, \Delta_Y)$ is also strongly $F$-regular.
\end{lem}

\begin{proof}
Take an open covering $\{ U_i \}_{i \in I}$ of $X$ such that $(U_i, \Delta|_{U_i})$ is globally $F$-regular for every $i \in I$.
It then follows from Lemma \ref{crepant} that the pair $(V_i:=f^{-1}(U_i), \Delta_Y|_{V_i})$ is globally $F$-regular for every $i$, which implies that $(Y, \Delta_Y)$ is strongly $F$-regular. 
\end{proof}

\subsection{Cartier isomorphism}
In this subsection, we briefly review the Cartier isomorphism. 
Let $X$ be an $n$-dimensional smooth (not necessarily projective) variety over a perfect field $k$ of characteristic $p>0$. Let 
\[
0 \to F_*\sO_X \xrightarrow{F_*d^0} F_*\Omega_X \xrightarrow{F_*d^1} F_*\Omega_X^2 \xrightarrow{F_*d^2} \cdots \xrightarrow{F_*d^{n-1}} F_*\omega_X \to 0
\]
be the push-forward of the de Rham complex of $X$ by the absolute Frobenius morphism $F:X \to X$. 
For each $i=1, \dots, n$, we define the $\sO_X$-submodules $B_X^i, Z_X^i$ of $F_* \Omega_X^i$ by 
\[
B_X^i:=\mathrm{Im}\; F_*d^{i-1}, \quad Z_X^i:=\mathrm{Ker}\; F_*d^i. 
\]
$B_X^i$ and $Z_X^i$ are locally free sheaves and we have the following short exact sequences. 
\begin{align*}
& 0 \to \sO_X \to F_*\sO_X \to B_X^1 \to 0,\\
& 0 \to Z_X^i \to F_*\Omega_X^i \to B_X^{i+1} \to 0,\\
& 0 \to B_X^i \to Z_X^i \xrightarrow{C} \Omega_X^i \to 0,
\end{align*}
where $C$ is the Cartier operator (see \cite{Il} for the details). 

\begin{defn}
Let $X$ be a normal variety over a perfect field $k$ of characteristic $p>0$. Let $U$ denote the smooth locus of $X$ and $\iota:U \hookrightarrow X$ denote the natural inclusion. For each $i=1, \dots, n$, the sheaf $\Omega_X^{[i]}$ of reflexive differential $i$-forms on $X$ is defined by $\Omega_X^{[i]}:=\iota_*\Omega_U^i$. Similarly, $B_X^{[i]}$ and $Z_X^{[i]}$ are defined by $B_X^{[i]}:=\iota_*B_U^{i}$ and  $Z_X^{[i]}:=\iota_*Z_U^i$, respectively.  
\end{defn}

\subsection{Symmetric, exterior, and divided powers}
Here, we present basic properties of divided powers which we will need in \S \ref{thickening section}. 

Suppose that $X$ is a scheme, $n$ is a nonnegative integer and $\mathcal{F}$ is a coherent sheaf on $X$.
Let $S^n(\mathcal{F})$ (resp. $\bigwedge^n(\mathcal{F})$) denote the $n$-th symmetric power (resp. the $n$-th exterior power) of $\mathcal{F}$.
Let $f: \mathcal{F} \to \mathcal{G}$ is a morphism of coherent sheaves on $X$, and then 
$f^{\otimes n} : \mathcal{F}^{\otimes n } \to \mathcal{G}^{\otimes n}$ induces 
the morphisms $S^n(f): S^n(\mathcal{F}) \to S^{n}(\mathcal{G})$ and $\bigwedge^n(f): \bigwedge^n(\mathcal{F}) \to \bigwedge^{n}(\mathcal{G})$. 

The natural morphism $\mathcal{F}^{\otimes n} \otimes \mathcal{F} \to \mathcal{F}^{\otimes n+1}$ induces the morphisms
\begin{eqnarray*}
m_{s}: S^n(\mathcal{F}) \otimes  \mathcal{F}&\to& S^{n+1}(\mathcal{F}) \textup{, and }\\
m_{e}:\bigwedge^n(\mathcal{F}) \otimes  \mathcal{F} &\to& \bigwedge^{n+1}(\mathcal{F}).
\end{eqnarray*}
Let $\mathcal{E}$ be a locally free sheaf of finite rank on $X$.
Then by patching together the comultiplication maps (\cite[Subsection 1.1]{Wey}), we define the morphism
\begin{eqnarray*}
\Delta_{e}  : \bigwedge^{n+1}(\mathcal{E}) &\to& \bigwedge^n(\mathcal{E}) \otimes \mathcal{E}.
\end{eqnarray*}

\begin{defn}
Let $f: \mathcal{E} \to \mathcal{F}$ be a morphism of locally free sheaves of finite rank on a scheme $X$.
For any integers $i,j \ge 0$, we define the morphism 
$\alpha_{f}^{i,j}: S^i(\mathcal{F}) \otimes \bigwedge^j(\mathcal{E}) \to S^{i+1}(\mathcal{F}) \otimes \bigwedge^{j-1}(\mathcal{E})$
 as the composition
\begin{equation*}
\xymatrix{
S^i(\mathcal{F}) \otimes \bigwedge^j(\mathcal{E})  \ar^-{\mathrm{id} \otimes \Delta_{e}}[rr] \ar_-{\alpha^{i,j}_f}[rrrrd] && S^i(\mathcal{F}) \otimes \mathcal{E} \otimes \bigwedge^{j-1}(\mathcal{E}) \ar^-{\mathrm{id} \otimes f \otimes \mathrm{id}}[rr] && S^i(\mathcal{F}) \otimes \mathcal{F} \otimes \bigwedge^{j-1}(\mathcal{E}) \ar^-{m_s \otimes \mathrm{id}}[d] \\
&&&& S^{i+1}(\mathcal{F}) \otimes \bigwedge^{j-1}(\mathcal{E}).}
\end{equation*}
\end{defn}

\begin{lem}\label{Koszul}
Let $X$ be a scheme and
\[
0 \to \mathcal{E} \xrightarrow{\ f\ } \mathcal{F} \xrightarrow{\ g\ } \mathcal{G} \to 0
\]
be an exact sequence of locally free sheaves of finite rank.
Let $n \ge 0$ be an integer and $r : = \min \{n, \mathrm{rank}(\mathcal{E}) \}$.
Then the sequence 
\[
0 \xleftarrow{\ \ } S^n(\mathcal{G}) \xleftarrow{S^n(g)} S^n(\mathcal{F}) \xleftarrow{\alpha^{m-1, 1}_f}  S^{n-1}(\mathcal{F}) \otimes \mathcal{E} \xleftarrow{\alpha^{m-2, 2}_f}  \cdots \xleftarrow{\alpha^{m-r,r}_f} S^{n-r}(\mathcal{F}) \otimes \bigwedge^r(\mathcal{E}) \xleftarrow{\ \ } 0
\]
is an exact sequence.
\end{lem}

\begin{proof}
By shrinking $X$, we may assume that $X=\Spec A$ is affine and $\mathcal{E}, \mathcal{F}$ and $\mathcal{G}$ are free.
Let $E$, $F$ and $G$ be the free $A$-modules corresponding to $\mathcal{E}, \mathcal{F}$ and $\mathcal{G}$, respectively.
Take a basis $e_1, \dots, e_u$ of $E$.
Since the short exact sequence 
\[
0 \to E \to F \to G \to 0
\]
splits, there exist $e_{u+1}, \dots, e_v \in F$ such that $e_1, \dots, e_u, e_{u+1}, \dots, e_v$ is a basis of $F$.
Let $R : = \oplus_{m\ge 0} S^m(F) \cong A[x_1, \dots, x_v]$ be the symmetric algebra of $F$.
Since $x_1, \dots, x_u \in R$ is a regular sequence, the Koszul complex
\[
0 \xleftarrow{} R/(x_1, \dots, x_u) \xleftarrow{} R \xleftarrow{} R \otimes_R E \xleftarrow{}  \cdots \xleftarrow{} R \otimes_R \bigwedge^u(E) \xleftarrow{} 0
\]
is exact.
Then the sequence in the assertion is exact, because it is isomorphic to the degree $n$ part of the above Koszul complex.
\end{proof}

Next, we consider the divided powers of locally free sheaves.
Let $X$ be a scheme, $n \ge 0$ be an integer and $\mathcal{E}$ be a locally free sheaf on $X$ of finite rank.
We define the $n$-th \textit{divided power} $D_n(\mathcal{E})$ by 
\[
D_n(\mathcal{E}):= (S^n(\mathcal{E^*}))^{*},
\]
where $(-)^{*} : = \mathcal{H}\mathrm{om}(- , \sO_X)$ is the $\sO_X$-dual.
We also define 
\[
\Delta_d := (m_s)^*: D_{n+1}(\mathcal{E}) \to D_n(\mathcal{E}) \otimes \mathcal{E}
\]
as the dual of $m_s: S^n(\mathcal{E}^*) \otimes \mathcal{E}^* \to S^{n+1}(\mathcal{E}^*)$. 
For a morphism $f : \mathcal{E} \to \mathcal{F}$ between locally free sheaves of finite rank, we write $D_n(f) : = (S^n(f^*))^* : D_n(\mathcal{E}) \to D_n(\mathcal{F})$.

\begin{defn}
Let $X$ be a scheme and $g : \mathcal{F} \to \mathcal{G}$ be a morphism of coherent sheaves on $X$. 
Suppose that $\mathcal{F}$ is a locally free sheaf of finite rank. 
For any integers $i,j \ge 0$, we define the morphism 
$\beta_{g}^{i,j}: D_i(\mathcal{F}) \otimes \bigwedge^j(\mathcal{G}) \to D_{i-1}(\mathcal{F}) \otimes \bigwedge^{j+1}(\mathcal{G})$
 as the composition
\[
\xymatrix{
D_i(\mathcal{F}) \otimes \bigwedge^j(\mathcal{G}) \ar^-{\Delta_{d} \otimes \mathrm{id}}[rr] \ar_-{\beta^{i,j}_g}[rrrrd] && D_{i-1}(\mathcal{F}) \otimes \mathcal{F} \otimes \bigwedge^{j}(\mathcal{G}) \ar^-{\mathrm{id} \otimes g \otimes \mathrm{id}}[rr] && D_{i-1}(\mathcal{F}) \otimes \mathcal{G} \otimes \bigwedge^{j}(\mathcal{G}) \ar^-{\mathrm{id} \otimes m_e}[d] \\
&&&& D_{i-1}(\mathcal{F}) \otimes \bigwedge^{j+1}(\mathcal{G}).
}
\]
\end{defn}

\begin{lem}[\textup{\cite[Proposition 1.1.2]{Wey}}]\label{wedge dual}
Let $X$ be a scheme and $\mathcal{E}$ be a locally free sheaf of rank $r$ on $X$.
Then for every $1 \le n \le r$, there is an isomorphism 
\[
\phi_n:  \bigwedge^n(\mathcal{E}^*)  \xrightarrow{\ \sim\ }(\bigwedge^n(\mathcal{E}))^{*}. \] 
Moreover, we have the following commutative diagram 
\[
\xymatrix{
\bigwedge^{n}(\mathcal{E}^*) \ar^-{\Delta_e}[r]  \ar_{\phi_n}[d] & \bigwedge^{n-1}(\mathcal{E}^*) \otimes \mathcal{E}^* . \ar^{\phi_{n-1} \otimes \mathrm{id}}[rd] & \\
(\bigwedge^n(\mathcal{E}))^* \ar^-{(m_e)^*}[r] & (\bigwedge^{n-1}(\mathcal{E}) \otimes \mathcal{E})^*  \ar^-{ \sim }[r] & (\bigwedge^{n-1}(\mathcal{E}))^* \otimes \mathcal{E}^*
}
\]

\end{lem}

\begin{prop}\label{divided cpx}
Let $X$ be a scheme, $\mathcal{E}, \mathcal{F}$ be locally free sheaves of finite rank and $\mathcal{G}$ be a coherent sheaf on $X$.
Suppose that there is an exact sequence
\[
0 \to \mathcal{E} \xrightarrow{\ f\ } \mathcal{F} \xrightarrow{\ g\ } \mathcal{G} \to 0.
\]
Then for every integer $n \ge 1$, the sequence
\begin{equation*}
\begin{split}
0 \to D_n(\mathcal{E}) \xrightarrow{D_n(f)} D_n(\mathcal{F}) \xrightarrow{\ \beta^{n,0}_g\ } D_{n-1}(\mathcal{F}) \otimes \mathcal{G}  \xrightarrow{\ \beta^{n-1,1}_g\ } & D_{n-2}(\mathcal{F}) \otimes \bigwedge^2 (\mathcal{G}) \xrightarrow{\ \beta^{n-2,2}_g\ } \cdots\\
 \cdots  \xrightarrow{\beta^{n-r+1,r-1}_g} & D_{n-r}(\mathcal{F}) \otimes \bigwedge^r (\mathcal{G}) \xrightarrow{\ \beta^{n-r,r}_g\ } \cdots
\end{split}
\end{equation*}
is a complex.
Moreover, this complex is exact if $\mathcal{G}$ is locally free.
\end{prop}

\begin{proof}
Since we have the commutative diagram
\[
\xymatrix{
D_n(\mathcal{E}) \ar^-{\Delta_d}[rr]  \ar_{D_n(f)}[d] & &D_{n-1}(\mathcal{E}) \otimes \mathcal{E}  \ar_{D_{n-1}(f) \otimes f}[d] \ar^{0}[drr]\\
D_n(\mathcal{F}) \ar^-{\Delta_d}[rr] \ar@/_12pt/[rrrr]_{\beta^{n,0}_g}  & & D_{n-1}(\mathcal{F}) \otimes \mathcal{F} \ar^-{\mathrm{id} \otimes g}[rr] && D_{n-1}(\mathcal{F}) \otimes \mathcal{G} ,
}
\]
the composite map $\beta^{n,0}_g \circ D_n(f)$ is the zero map.
We next show that the composition 
\[
\beta^{i-1,j+1}_g \circ \beta^{i,j}_g : D_i(\mathcal{F}) \otimes \bigwedge^j (\mathcal{G}) \to D_{i-2}(\mathcal{F}) \otimes \bigwedge^{j+2} (\mathcal{G})
\]
is the zero map for every integers $i \ge 2$ and $j \ge 0$.
Since the map $\beta \circ \beta$ factors through the composition 
\[
D_i(\mathcal{F}) \otimes \bigwedge^j (\mathcal{G}) \xrightarrow{(\Delta_d \otimes \mathrm{id}) \circ \Delta_d \otimes \mathrm{id}} D_{i-2}(\mathcal{F}) \otimes \mathcal{F} \otimes \mathcal{F} \otimes \bigwedge^j (\mathcal{G}) \xrightarrow{\mathrm{id} \otimes m_e \otimes \mathrm{id}} D_{i-2}(\mathcal{F}) \otimes \bigwedge^2(\mathcal{F}) \otimes \bigwedge^j (\mathcal{G}),\]
it is enough to show that 
\[
(\mathrm{id} \otimes m_e) \circ (\Delta_d \otimes \mathrm{id})  \circ \Delta_d: D_i(\mathcal{F}) \to D_{i-2}(\mathcal{F}) \otimes \bigwedge^2(\mathcal{F})\]
 is the zero map.
After shrinking $X$, we may assume that $X$ is affine and $\mathcal{F}$ is free.
Then we can verify the equation $(\mathrm{id} \otimes m_e) \circ (\Delta_d \otimes \mathrm{id})  \circ \Delta_d=0$ by looking at a canonical basis of $D_n(\mathcal{F})$ (see \cite[Proposition 1.1.7 (c)]{Wey}), which proves that the sequence in the assertion is a complex.

Finally, we assume that $\mathcal{G}$ is locally free.
Applying Lemma \ref{Koszul} to the exact sequence
\[
0 \to \mathcal{G}^* \xrightarrow{\ g* \ } \mathcal{F}^* \xrightarrow{\ f* \ } \mathcal{E}^* \to 0,
\]
we obtain the exact sequence 
\[
0 \xleftarrow{\ \ } S^n(\mathcal{E}^*) \xleftarrow{S^n(f^*)} S^n(\mathcal{F}^*) \xleftarrow{\ \alpha\ } S^{n-1}(\mathcal{F}^*) \otimes \mathcal{G}^* \xleftarrow{\ \alpha\ }  \cdots \xleftarrow{\ \alpha\ } S^{n-r}(\mathcal{F}^*) \otimes \bigwedge^r(\mathcal{G}^*) \xleftarrow{\ \ } 0
\]
of locally free sheaves, where $r:=\min\{n, \mathrm{rank}(\mathcal{G)}\}$. 
It follows from Lemma \ref{wedge dual} that the sequence in the assertion is isomorphic to the dual of the above exact sequence, which completes the proof.
\end{proof}

The following lemma is well-known to experts, but we provide its proof here for the convenience of the reader. 

\begin{lem}\label{div sym}
Let $X$ be a scheme over a field $k$ and $\mathcal{F}$ be a locally free sheaf of finite rank on $X$.
Assume that one of the following conditions is satisfied. 
\begin{enumerate}
\item[\textup{(a)}] $k$ is of characteristic zero,  
\item[\textup{(b)}] $k$ is of characteristic $p>m$, 
\item[\textup{(c)}] $\mathcal{F}$ is a direct sum of line bundles. 
\end{enumerate}
Then 
$D_m(\mathcal{F}) \cong S^m(\mathcal{F})$.
\end{lem}

\begin{proof}
We first assume that $\mathrm{ch}(k)=0$ or $\mathrm{ch}(k)>m$.
In this case, by considering the paring $S^m(\mathcal{F}^*) \otimes S^m(\mathcal{F}) \to \sO_X$ which sends a local section $f_1 \cdots f_m \otimes x_1 \cdots x_m$ to $\sum_{\sigma \in \mathfrak{S}_m} f_{\sigma(1)}(x_1) \cdots f_{\sigma(m)} (x_m)$, we obtain the morphism $\phi : S^m(\mathcal{F}^*) \to (S^m(\mathcal{F}))^*$.
By the assumption on the characteristic of $k$, we can see that $\phi$ is surjective.
Since $S^m(\mathcal{F}^*)$ and $(S^m(\mathcal{F}))^*$ are locally free sheaves of the same rank, the morphism $\phi$ is isomorphism, which proves $D_m(\mathcal{F}) \cong S^m(\mathcal{F})$.

If $\mathcal{F}$ is a direct sum of line bundles $\oplus_{i=1}^n L_i$, then 
\begin{eqnarray*}
D_m( \mathcal{F}) \cong (S^m( \oplus_{i=1}^n L_i^*))^* &\cong& \bigoplus_{\lambda_1 + \cdots \lambda_n=m } ((L_{1}^*)^{\otimes \lambda_1} \otimes \cdots \otimes (L_{n}^*)^{\otimes \lambda_n})^*\\
&\cong& \bigoplus_{\lambda_1 + \cdots \lambda_n=m } L_{1}^{\otimes \lambda_1} \otimes \cdots \otimes L_{n}^{\otimes \lambda_n}\\ 
&\cong& S^m( \mathcal{F}).
\end{eqnarray*}
\end{proof}

\subsection{Algebraic spaces}
In this subsection, we quickly review the basic notions of algebraic spaces that will be used in the proof of Theorem \ref{smoothing}. 
The reader is referred to \cite[Section 5]{Ol} for the definition and basic properties of algebraic spaces.

For a scheme $X$ over a field $k$, $(\mathrm{Sch}/k)$ denotes the \textit{big \'{e}tale site} of $\Spec k$ and $h_X$ denotes the functor 
\[
h_X : = \Hom_k(-, X): (\mathrm{Sch}/k)^{\mathrm{op}} \to (\mathrm{Set})
\]
from the opposite category of $(\mathrm{Sch}/k)$ to the category of sets ($\mathrm{Set})$.
We note by \cite[Theorem 4.1.2]{Ol} that the functor $h_X$ is an \textit{algebraic space} over $k$.
By Yoneda Lemma, the rule $X \mapsto h_X$ defines a fully faithful functor from $(\mathrm{Sch}/k)$ to the category of algebraic spaces over $k$.
Therefore, we often identify a scheme $X$ with the algebraic space $h_X$.
Let $G : (\mathrm{Sch}/k)^{\mathrm{op}} \to (\mathrm{Set})$ be an algebraic space over $S$ and $|G|$ denotes the \textit{underlying topological space} of $G$ (see \cite[Definition 03BY]{Sta}).

\begin{rem}
Since an algebraic space is automatically a sheaf in the fppf topology (see \cite[Theorem 5.2.2]{Ol} or \cite[Lemma 076M]{Sta}), the definition of algebraic spaces in \cite{Ol} is the same as that in \cite[Definition 025Y]{Sta}.
Moreover, it follows from \cite[Lemma 03K4]{Sta} that if an algebraic space $G$ is decent (see \cite[Definition 03I8]{Sta}), then the definition of the underlying topological space of $G$ in \cite[6.3.3]{Ol} is the same as that in \cite{Sta}.
We also note that a separated algebraic space is decent by \cite[Subsection 03I7]{Sta}.
\end{rem}

A morphism $f : G \to H$ of algebraic spaces induces a continuous morphism $|f|: |G| \to |H|$.
If $f$ is an immersion (resp. open immersion, closed immersion), then so is $|f|$ by \cite[Lemma 04CD]{Sta}.

\begin{lem}\label{univ inj}
Let $f: H \to G$ be a morphism of algebraic spaces over a field $k$, $\{ G_j \}_{j \in J}$ be a set of algebraic spaces and $i_j: G_j \to G$ be immersions such that 
\[
|G|=\bigcup_{j \in J} |i_j|(|G_j|).
\] 
If the induced morphism $f_j:f^{-1}(G_j) \to G_j$ is surjective $($resp. universally injective$)$ $($see \cite[Definition 03ME, 03MV]{Sta}$)$ for every $j \in J$, then so is $f$.
\end{lem}

\begin{proof}
The assertion follows from \cite[Lemma 03H4]{Sta}.
\end{proof}

Let $(G_i^{\circ} \subset G_i, f_i: H_i \to G_i)$ be \textit{elementary distinguished squares} (see \cite[Definition 08GM]{Sta}) for $i=1,2$ and $H_i^{\circ} \subset H_i$ be the fiber product $G_i^{\circ} \times_{G_i} H_i$.
\[
\xymatrix{
H_i^{\circ} \ar[r] \ar[d] & H_i \ar^-{f_i}[d]\\
G_i^{\circ} \ar[r] & G_i
}\]
For $i=1,2$, let $T_i \subset G_i$ be the reduced closed subspace such that $|T_i|=|G_i| \setminus |G_i^{\circ}|$ and let $T_i' \subset H_i$ be the reduced closed subspace such that $|T_i'| = |H_i| \setminus |H_i^{\circ}|$.
We remark that such subspaces exist uniquely by \cite[Lemma 03IQ]{Sta}.
It follows from the definition of elementary distinguished squares that $T_i' \cong f_i^{-1}(T_i)$ for $i=1,2$.

Let $\phi_G: G_1 \to G_2$ and $\phi_H : H_1 \to H_2$ be morphisms of algebraic spaces such that $\phi_G \circ f_1 = f_2 \circ \phi_H$.
We further assume that $|\phi_G| (|G_1^{\circ}|) \subset |G_2^{\circ}|$ and $|\phi_H|(|H_1^{\circ}|) \subset |H_2^{\circ}|$.
By \cite[Lemma 03IE]{Sta}, we have the following diagram in which each square is commutative and the front and back squares are Cartesian.
\begin{equation}\label{diagram1}
\vcenter{
\xymatrix{
& H_1^{\circ} \ar[rr] \ar[dd]|(.5)\hole \ar[ld] && H_1 \ar^-{f_1}[dd] \ar^-{\phi_H}[ld]\\
H_2^{\circ} \ar[rr] \ar[dd]  && H_2 \ar^(.4){f_2}[dd]  & \\
& G_1^{\circ} \ar[rr]|(.5)\hole \ar[ld]&& G_1 \ar^-{\phi_G}[ld]\\
G_2^{\circ} \ar[rr] && G_2 &
}}
\end{equation}

\begin{prop}\label{cube}
With the above notation, we further assume that the top and left squares in the diagram \eqref{diagram1} are Cartesian.
Then the remaining two squares, the bottom and right squares, are also Cartesian.
\end{prop}

\begin{proof}
We first look at the bottom square
\[
\xymatrix{
G_1^{\circ} \ar[r] \ar[d] & G_1 \ar^-{\phi_G}[d]\\
G_2^{\circ} \ar[r] & G_2.
}\]
Since the top square is Cartesian, it follows from \cite[Lemma 03H4]{Sta} that $|\phi_H|^{-1}(|T_2'|)=|T_1'|$.
Combining with the fact that $f_i: T_i' \to T_i$ is isomorphic for $i=1,2$, we have $|\phi_G|^{-1}(|T_2|)=|T_1|$, which implies $|\phi_G|^{-1}(|G_2^{\circ}|)=|G_1^{\circ}|$.
Therefore, the bottom square is Cartesian by \cite[Lemma 03BZ]{Sta}.

We next look at the right square in the diagram \eqref{diagram1}.
Let $\tilde{H} : = G_1 \times_{G_2} H_2$ denote the fiber product and $\psi: H_1 \to \tilde{H}$ denote the induced morphism.
\begin{equation}\label{diagram2}
\vcenter{
\xymatrix{
H_1 \ar^-{\psi}[r] \ar_-{f_1}[rd] & \tilde{H}  \ar^-{\mathrm{pr}_1}[d] \ar^-{\mathrm{pr}_2}[r] & H_2 \ar^-{f_2}[d]\\
& G_1 \ar^-{\phi_G}[r] & G_2
}}
\end{equation}
By \cite[Lemma 05W5]{Sta}, it suffices to show that $\psi$ is \'{e}tale, surjective and universally injective.
Let $\tilde{T} : = \mathrm{pr}_1^{-1}(T_1)$ and $\tilde{H}^{\circ} : = \mathrm{pr}_1^{-1}(G_1^{\circ})$ denote the fiber products and $\psi_T: T_1' \to \tilde{T}$ and $\psi^{\circ}: H_1^{\circ} \to \tilde{H}^{\circ}$ denote the morphisms induced by $\psi: H_1 \to \tilde{H}$.
Then the diagram \eqref{diagram2} induces a Cartesian diagram
\[
\xymatrix{
 \tilde{H}^{\circ}  \ar^-{}[d] \ar^-{}[r] & H_2^{\circ} \ar^-{}[d]\\
G_1^{\circ} \ar^-{}[r] & G_2^{\circ}. 
}\]
Combining with the assumption that the left square in the diagram \eqref{diagram1} is Cartesian, we  see that the morphism $\psi^{\circ}: H_1^{\circ} \to \tilde{H}^{\circ}$ is isomorphic.
On the other hand, since $|\phi_G|(|T_1|) \subseteq |T_2|$, by \cite[Lemma 03JJ]{Sta}, the diagram \eqref{diagram2} induces another diagram
\[
\xymatrix{
T_1' \ar^-{\psi_T}[r] \ar[rd] & \tilde{T} \ar^-{}[d] \ar^-{}[r] & T_2' \ar^-{}[d]\\
& T_1 \ar^-{\phi_G}[r] & T_2
}\]
in which the square is Cartesian.
The morphisms $T_1' \to T_1$ and $T_2' \to T_2$ in the above diagram are both isomorphic, and then so is $\psi_T$. 
By the definition of elementary distinguished squares and the fact that the back square in the diagram \eqref{diagram1} is Cartesian, $\psi^{-1}(\tilde{T})=T_1'$ and $\psi^{-1}(\tilde{H}^{\circ})=H_1^{\circ}$. 
It then follows from Lemma \ref{univ inj} that $\psi$ is surjective and universally injective.
Finally, since $f_1$ and $f_2$ are both \'{e}tale, applying \cite[Lemma 0466, 05W3]{Sta} to the diagram \eqref{diagram2}, we conclude that $\psi$ is \'{e}tale.
\end{proof}

From now on, all algebraic spaces are assumed to be separated and of finite type over a base field $k$.

We next discuss the abelian category $\mathrm{Qcoh}(G)$ of quasi-coherent sheaves on the \textit{small \'{e}tale site} $\mathrm{Et}(G)$ of an algebraic space $G$ over a field $k$. 
The reader is referred to \cite[Section 7]{Ol} for the definition and basic properties of quasi-coherent sheaves.
The \textit{tangent sheaf} $T_G$ of $G$ is defined as $T_G := (\Omega_{G/k})^{*} : = \mathcal{H}om_{\sO_G} (\Omega_{G/k}, \sO_G)$,
where $\sO_G$ and $\Omega_{G}=\Omega_{G/k}$ are the \textit{structure sheaf} and the \textit{cotangent sheaf} of $G$, respectively. 
When $G$ is normal, the \textit{canonical sheaf} $\omega_G$ of $G$ is defined as $\omega_G := (\bigwedge^{\dim |{G}|} \Omega_{G/k})^{**}$. 
Similarly, for a \textit{Deligne-Mumford stack} $\mathcal{X} \to (\mathrm{Sch}/k)$, $\mathrm{Qcoh}(\mathcal{X})$ denotes the abelian category of quasi-coherent sheaves on the \'{e}tale site $\mathrm{\textup{\'{E}}t}(\mathcal{X})$ of $\mathcal{X}$ (see \cite[Sections 8, 9]{Ol} for details).
Note 
that for an algebraic space $G$, the natural functor $\mathrm{Et}(G) \to \mathrm{\textup{\'{E}}t}(\mathcal{S}_G)$ is equivalent, 
where $\mathcal{S}_G$ is the Deligne-Mumford stack associated to $G$ (see \cite[3.2.7]{Ol}, \cite[Lemma 03YS]{Sta}).
Therefore, the natural functor $\mathrm{Qcoh}(\mathcal{S}_G) \to \mathrm{Qcoh}(G)$ is also equivalent.

Let $G, H$ be integral algebraic spaces of finite type over a field $k$ (see \cite[Definition 0AD4]{Sta}) and $f: G \to H$ be a birational morphism (see \cite[Definition 0ACV]{Sta}).
Let $U \subset H$ denote the maximal open subspace such that the induced morphism $f^{-1}(U) \to U$ is isomorphic.
The \textit{exceptional locus} of $f$ is defined as the closed subset $\mathrm{Exc}(f) : =|G| \setminus |f^{-1}(U)| \subseteq |G|$.
We say that $f$ is \textit{small} if $\dim |G|- \dim (\mathrm{Exc}(f)) \ge 2$.

\begin{lem}\label{space small}
Let $f: G \to H$ be a proper small birational morphism of normal integral algebraic spaces of finite type over a perfect field $k$.
Then the following hold.
\begin{enumerate}
\item $f_* T_G \cong T_H$.
\item $R^i f_* \mathcal{F} =0$ for every integer $i \ge \dim |G| -1$ and every coherent sheaf $\mathcal{F}$ on $G$.
\end{enumerate}
\end{lem}

\begin{proof}
(1) Let $U \subset H$ be the largest open subspace such that $f|_U: f^{-1}(U) \to U$ is isomorphic.
Let $i: f^{-1}(U) \hookrightarrow G$ and $j: U \hookrightarrow H$ denote natural inclusions.
Noting that $T_G$ is reflexive, that is, the natural morphism $T_G \to (T_G)^{**}$ is isomorphic, we see that the morphism $T_G \to i_* T_{f^{-1}(U)}$ is isomorphic.
Similarly, the morphism $T_H \to j_*T_{U}$ is also isomorphic. 
Therefore, we have the isomorphisms 
\[f_* T_G \cong f_* i_* T_{f^{-1}(U)} \cong j_* T_U \cong T_H.\]
(2) The assertion follows from \cite[Lemma 0A4T]{Sta} and \cite[Lemma 0A4J]{Sta}. 
\end{proof}

Let $k$ be a perfect field of characteristic $p>0$ and $F_k : \Spec k \to \Spec k$ denote the Frobenius morphism.
Given a $k$-scheme $Y$ with structure morphism $\pi: Y \to \Spec k$, 
let $Y^{(1)}$ be the $k$-scheme identified with $Y$ as a scheme but with structure morphism $F_k \circ \pi: Y \to \Spec k$.
Note that the rule $Y \mapsto Y^{(1)}$ induces a functor $\tilde{F_k} : (\mathrm{Sch}/k) \to (\mathrm{Sch}/k)$, which is isomorphic since $k$ is perfect.
Moreover, the Frobenius morphism $F: Y^{(1)} \to Y$ defines a natural transformation $F: \tilde{F_k} \to \mathrm{id}$.
Given an algebraic space $G$ over $k$, we define an algebraic space $G^{(1)}$ as the composite functor 
\[G^{(1)} : = G \circ \tilde{F_k} : ( \mathrm{Sch}/k) \to (\mathrm{Set}).\]
We define the \textit{relative Frobenius morphism} of $G$ as the morphism $F_G : G \to G^{(1)}$ of algebraic spaces induced by the natural transformation $F: \tilde{F_k} \to \mathrm{id}$.
We note that for a $k$-scheme $X$, the morphism $F_{h_X}: h_X \to (h_X)^{(1)}$ coincides with the relative Frobenius morphism $F: X \to X' : =X \times_{k, F_k} k$ via the identification of $(h_X)^{(1)}$ with $h_{X'}$.
We say that $G$ is \textit{globally $F$-split} if the natural morphism 
\[
\sO_G \to (F_G)_* \sO_{G^{(1)}}
\]
splits as an $\sO_G$-module-homomorphism.
We note that a $k$-scheme $X$ is globally $F$-split if and only if so is $h_X$, because $k$ is perfect.

\begin{lem}\label{space small GFS}
Let $f: G \to H$ be a proper small birational morphism of normal integral algebraic spaces over a perfect field of characteristic $p>0$.
Then $G$ is globally $F$-split if and only if so is $H$.
\end{lem}

\begin{proof}
The proof is very similar to those of \cite[Proposition 4]{MR} and \cite[Lemma 2.14]{GOST}.
\end{proof}

\section{Akizuki-Nakano vanishing for globally \texorpdfstring{$F$-split}{F-split} 3-folds}\label{vanishing section}
Since any smooth globally $F$-split variety over a perfect field $k$ of characteristic $p>0$ lifts to $W_2(k)$ (see \cite[p.164]{Il2} or \cite[Corollary 9.2]{Ki}), by a result of Raynaud (see \cite[Corollaire 2.8]{DI}), 
it satisfies the Akizuki-Nakano vanishing theorem if $p \ge \dim X$. 
In particular, the Akizuki-Nakano vanishing theorem holds on smooth $F$-split surfaces. 
Kawakami recently proved that it also holds on possibly singular globally $F$-split surfaces, using Graf's extension theorem \cite{Gr}. 

\begin{thm}[cf.~\textup{\cite[Corollary 4.8]{Kaw}}]\label{dim 2}
Let $X$ be a normal globally $F$-split projective surface over a perfect field of characteristic $p>0$ and $\mathcal{H}$ be an ample line bundle on $X$. 
If $X$ has only strongly $F$-regular singularities, then for every nonnegative integers $i, j$ with $i+j<2$, one has 
\[H^i(X, \Omega_X^{[j]} \otimes \mathcal{H}^{-1})=0.\]
\end{thm}
\begin{proof}
The case where $i=0$ follows from \cite[Corollary 4.8]{Kaw} and the case where $j=0$ does from Proposition \ref{kodaira}. 
\end{proof}

In this section, we study the Akizuki-Nakano vanishing theorem for globally $F$-split 3-folds. 

\begin{prop}\label{vanishing (2,0)}
Let $X$ be a normal globally $F$-split proper variety of dimension $n \ge 2$ over a perfect field of characteristic $p>0$ and $\mathcal{F}$ be a globally generated line bundle on $X$ such that the morphism $\Phi_{\mathcal{F}}:X \to \mathbb{P}^N=\mathbb{P}(H^0(X,\mathcal{F})^*)$ induced by $\mathcal{F}$ is generically \'etale on its image. 
Let $\mathcal{G}$ be a line bundle on $X$ such that $H^{n-1}(X, \mathcal{F}\otimes \mathcal{G})=H^n(X, \mathcal{G})=0$. 
Then 
\[H^0(X, \Omega_X^{[n-1]} \otimes (\mathcal{F} \otimes \mathcal{G})^{-1})=0.\] 
\end{prop}
\begin{proof}
Let $U$ be the smooth locus of $X$ and $\iota:U \hookrightarrow X$ be a natural inclusion.
Since 
\[
\Omega_X^{[n-1]}=\iota_*\Omega_U^{n-1} \cong \iota_*\mathcal{H}\mathrm{om}(\Omega_U^1, \omega_U) \cong \mathcal{H}\mathrm{om}(\Omega_X^{[1]}, \omega_X), 
\] 
one has the isomorphisms
\begin{align*}
H^0(X, \Omega_X^{[n-1]} \otimes (\mathcal{F} \otimes \mathcal{G})^{-1}) 
&\cong \mathrm{Hom}(\mathcal{F} \otimes \mathcal{G}, \mathcal{H}\mathrm{om}(\Omega_X^{[1]}, \omega_X))\\
&\cong \mathrm{Hom}(\Omega_X^{[1]} \otimes (\mathcal{F} \otimes \mathcal{G}), \omega_X). 
\end{align*}
Thus, by Serre duality (which holds for top cohomology without the Cohen-Macaulay assumption), it is enough to show that $H^n(X, \Omega_X^{[1]} \otimes (\mathcal{F} \otimes \mathcal{G}))=0$. 

Pulling back the Euler sequence on $\mathbb{P}^N$ twisted by $\mathcal{O}_{\mathbb{P}^N}(1)$ via $\Phi_{\mathcal{F}}$ and tensoring it with $\mathcal{G}$, one has the exact sequence 
\[
0 \to \Phi_{\mathcal{F}}^*\Omega_{\mathbb{P}^N}(1) \otimes \mathcal{G} \to \mathcal{G}^{\oplus N} \to \mathcal{F} \otimes \mathcal{G} \to 0.
\]
It follows from the assumption $H^{n-1}(X, \mathcal{F} \otimes \mathcal{G})=H^n(X, \mathcal{G})=0$ that 
\[H^n(X, \Phi_{\mathcal{F}}^*\Omega_{\mathbb{P}^N}(1) \otimes \mathcal{G})=0.\] 
On the other hand, the first exact sequence 
\[\Phi_{\mathcal{F}}^*\Omega_{\mathbb{P}^N}(1) \otimes \mathcal{G} \to \Omega_X \otimes (\mathcal{F} \otimes \mathcal{G}) \to \Omega_{X/\mathbb{P}^N} \otimes (\mathcal{F} \otimes \mathcal{G}) \to 0\] twisted by $\mathcal{F} \otimes \mathcal{G}$ induces the exact sequence 
\[
H^n(X, \Phi_{\mathcal{F}}^*\Omega_{\mathbb{P}^N}(1) \otimes \mathcal{G}) \to H^n(X, \Omega_X \otimes (\mathcal{F} \otimes \mathcal{G})) \to H^n(X, \Omega_{X/\mathbb{P}^N} \otimes (\mathcal{F} \otimes \mathcal{G})) \to 0. 
\]
Since $\Phi_{\mathcal{F}}$ is generically \'etale on its image, the support of $\Omega_{X/\mathbb{P}^N}$ is a proper closed subset of $X$ and, in particular, $H^n(X, \Omega_{X/\mathbb{P}^N} \otimes (\mathcal{F} \otimes \mathcal{G}))=0$. 
Therefore, $H^n(X, \Omega_X \otimes (\mathcal{F} \otimes \mathcal{G}))=0$. 
Noting that the natural map $\Omega_X \to \Omega_X^{[1]}$ is an isomorphism on $U$, that is, an isomorphism in codimension one, we can conclude that 
\[H^n(X, \Omega_X^{[1]} \otimes ( \mathcal{F} \otimes \mathcal{G})) \cong H^n(X, \Omega_X \otimes (\mathcal{F} \otimes \mathcal{G}))=0.\]
\end{proof}

\begin{prop}\label{vanishing (1,1)}
Let $X$ be a normal globally $F$-split proper 3-fold over a perfect field of characteristic $p>0$ which has only isolated singularities. 
Let $\mathcal{L}$ be a line bundle on $X$, and suppose there exists an integer $1 \le i_0 \le p$ such that $\mathcal{L}^{\otimes i_0}$ is globally generated and the morphism $\Phi_{\mathcal{L}^{\otimes i_0}}:X \to \mathbb{P}^{N}=\mathbb{P}(H^0(X,\mathcal{L}^{\otimes i_0})^*)$ induced by $\mathcal{L}^{\otimes i_0}$ is generically \'etale on its image.
Assume in addition the following three conditions. 
\begin{enumerate}
\item [$\textup{(i)}$] $H^1(X, \Omega_X^{[1]} \otimes \mathcal{L}^{\otimes (-p^e)})=0$ for all sufficiently large $e$. 
\item[$\textup{(ii)}$] $H^{2}(X, \mathcal{L}^{\otimes (\pm p^e)})=0$ for every integer $e \ge 1$. 
\item[$\textup{(iii)}$] $H^3(X, \mathcal{L}^{\otimes (p-i_0)p^e})=0$ for every integer $e \ge 0$. 
\end{enumerate}
Then $H^1(X, \Omega_X^{[1]} \otimes \mathcal{L}^{-1})=0$. 
\end{prop}
\begin{proof}
By the condition (i), we may assume that $H^1(X, \Omega_X^{[1]} \otimes \mathcal{L}^{\otimes (-p)})=0$. 
Let $U$ be the smooth locus of $X$ and $\iota:U \hookrightarrow X$ be a natural inclusion. 
The exact sequence $0 \to B^1_U \to Z^1_U \to \Omega^1_U \to 0$ induces an inclusion $Z^{[1]}_X/B^{[1]}_X \to \Omega_X^{[1]}$, which is an isomorphism on $U$, that is, an isomorphism in codimension one. 
Thus, it suffices to prove that $H^1(X, (Z^{[1]}_X/B^{[1]}_X) \otimes \mathcal{L}^{-1})=0$. 

Since $X$ is globally $F$-split, $0 \to \mathcal{O}_U \to F_*\mathcal{O}_U \to B^1_U \to 0$ is a split exact sequence. 
Pushing it forward by $\iota$, we can see that $0 \to \mathcal{O}_X \to F_*\mathcal{O}_X \to B^{[1]}_X \to 0$ is also a split exact sequence, and the condition (ii) implies that $H^2(X, B^{[1]}_X \otimes  \mathcal{L}^{-1})=0$. 
Therefore, it is enough to show that $H^1(X, Z^{[1]}_X \otimes \mathcal{L}^{-1})=0$. 

The exact sequence $0 \to Z^1_U \to F_*\Omega^1_U \to B^2_U \to 0$ induces an exact sequence 
\begin{equation}
0 \to Z^{[1]}_X \to F_*\Omega^{[1]}_X \to B^{[2]}_X, \tag{$\star$}
\end{equation}
and by assumption, $H^1(X, F_*\Omega^{[1]}_X \otimes \mathcal{L}^{-1}) \cong H^1(X, \Omega^{[1]}_X \otimes \mathcal{L}^{\otimes (-p)})=0$. 
On the other hand, by the choice of $i_0$, the morphism $\Phi_{\mathcal{L}^{\otimes i_0p^e}}:X \to \mathbb{P}^{N_e}=\mathbb{P}(H^0(X,\mathcal{L}^{\otimes i_0p^e})^*)$ induced by $\mathcal{L}^{\otimes i_0p^e}$ is generically \'etale on its image for every integer $e \ge 0$. 
It then follows from the conditions (ii), (iii) and Proposition \ref{vanishing (2,0)}, by substituting $\mathcal{F}=\mathcal{L}^{\otimes i_0}$ and $\mathcal{G}=\mathcal{L}^{\otimes (p-i_o)}$, that $H^0(X, F_*\Omega_X^{[2]} \otimes \mathcal{L}^{-1}) \cong H^0(X, \Omega_X^{[2]} \otimes \mathcal{L}^{\otimes (-p)})=0$. 
Since 
\[B^{[2]}_X=\iota_*B^2_U \subset \iota_*F_*\Omega^2_U \subset F_*\Omega^{[2]}_X,\]
we have $H^0(X, B_X^{[2]} \otimes \mathcal{L}^{-1})=0$. 
Thus, by the above exact sequence $(\star)$, we obtain that $H^1(X, Z^{[1]}_X \otimes \mathcal{L}^{-1})=0$. 
\end{proof}

We show an algebraic space version of Proposition \ref{vanishing (1,1)} that will be needed in the proof of Proposition \ref{relative vanishing}. 
\begin{prop}\label{vanishing on stack}
Let $G$ be a $3$-dimensional globally $F$-split integral algebraic space smooth and proper over a perfect field of characteristic $p>0$ and $\mathcal{L}$ be a line bundle on $G$.
Suppose that there exists an integer $1 \le i_0 \le p$ and a generically \'{e}tale morphism $\Phi: G \to Z$ to an algebraic space represented by a projective variety $Z \subseteq \PP^N_k$ such that $\mathcal{L}^{\otimes i_0} \cong \Phi^* (\sO_{\PP^N}(1)|_Z)$.
Assume in addition the following three conditions. 
\begin{enumerate}
\item [$\textup{(i)}$] $H^1(G, \Omega_{G}^{1} \otimes \mathcal{L}^{\otimes (-p^e)})=0$ for all sufficiently large $e$. 
\item[$\textup{(ii)}$] $H^{2}(G, \mathcal{L}^{\otimes (\pm p^e)})=0$ for every integer $e \ge 1$. 
\item[$\textup{(iii)}$] $H^3(G, \mathcal{L}^{\otimes (p-i_0)p^e})=0$ for every integer $e \ge 0$. 
\end{enumerate}
Then $H^1(G, \Omega_{G}^{1} \otimes \mathcal{L}^{-1})=0$. 
\end{prop}

\begin{proof}
First note by \cite[Theorem 2.22]{Ni} that Serre duality holds for coherent sheaves on a smooth proper algebraic space over a perfect field.
Moreover, it follows from \cite[Corollary 3.3.17]{Ol2} that the Cartier isomorphism can be generalized to the context of algebraic spaces. 
Thus, the proof is essentially the same as that of Proposition \ref{vanishing (1,1)}.
\end{proof}

\begin{lem}\label{reflexive serre vanishing}
Let $X$ be a normal projective variety of dimension $n \ge 2$ and $\mathcal{L}$ be an ample line bundle on $X$. 
If $\mathcal{F}$ is a coherent reflexive $\sO_X$-module, then for all sufficiently large $m$, we have 
\[H^1(X, \mathcal{F} \otimes \mathcal{L}^{\otimes (-m)})=0.\] 
\end{lem}

\begin{proof}

Take a surjection $f: \mathcal{E} \twoheadrightarrow \mathcal{F}^{*}$ from a locally free sheaf $\mathcal{E}$ of finite rank to the dual $\mathcal{F}^{*} : = \mathcal{H}\mathrm{om}(\mathcal{F}, \sO_X)$ of $\mathcal{F}$, 
and consider the exact sequence
\[
0 \to \mathcal{F}^{**} \xrightarrow{f^{*}} \mathcal{E}^{*} \to C \to 0,
\]
where $C$ is the cokernel of $f^{*}$.
Since $(\mathrm{Ker}\; f)^{*}$ is torsion free, so is the subsheaf $C \subseteq (\mathrm{Ker}\; f)^{*}$, which implies $H^0(X, C \otimes \mathcal{L}^{\otimes (-m)})=0$ for all sufficiently large $m$.
On the other hand, since $\mathcal{E}^{*}$ is locally free, it follows from the lemma of Enriques-Severi-Zariski \cite[III. Corollary 7.8]{Hart} that $H^1(X, \mathcal{E}^{*} \otimes \mathcal{L}^{\otimes (-m)})=0$ for all sufficiently large $m$.
Thus, by the above exact sequence, we obtain that 
\[
H^1(X, \mathcal{F} \otimes \mathcal{L}^{\otimes (-m)}) \cong H^1(X, \mathcal{F}^{**} \otimes \mathcal{L}^{\otimes (-m)}) =0
\] for all sufficiently large $m$.
\end{proof}

We now prove that a weak form of the Akizuki-Nakano vanishing theorem holds on globally $F$-regular 3-folds with isolated singularities. 
\begin{thm}\label{AN vanishing}
Let $X$ be a normal Cohen-Macaulay projective globally $F$-split 3-fold over a perfect field of characteristic $p>0$ and $\mathcal{H}$ be a globally generated ample line bundle on $X$ such that the morphism $\Phi_{\mathcal{H}}:X \to \mathbb{P}^{N}=\mathbb{P}(H^0(X,\mathcal{H})^*)$ induced by $\mathcal{H}$ is generically \'etale on its image. 
Suppose that $H^0(X, \omega_X)=0$ and $X$ has only isolated singularities. 
Then for all nonnegative integers $i,j$ with $i+j<3$, one has 
\[H^i(X, \Omega_X^{[j]} \otimes \mathcal{H}^{-1})=0.\]
\end{thm}
\begin{proof}
The case where $j=0$ is a special case of Proposition \ref{kodaira}. 
The case where $(i,j)=(0,1)$ follows from essentially the same argument as the proof of \cite[Proposition 4.6]{Kaw}. 
The case where $(i,j)=(0,2)$ is immediate from Propositions \ref{kodaira} and \ref{vanishing (2,0)} by substituting $\mathcal{F}=\mathcal{H}$ and $\mathcal{G}=\mathcal{O}_X$. 
Finally, the case where $(i,j)=(1,1)$ follows from Proposition \ref{vanishing (1,1)} by substituting $\mathcal{L}=\mathcal{H}$ and $i_0=1$.
Indeed, in this case, the condition (i) in Proposition \ref{vanishing (1,1)} is verified by Lemma \ref{reflexive serre vanishing} and the conditions (ii) and (iii) are done by Proposition \ref{kodaira}. 
\end{proof}

\begin{rem}\label{GFR case1}
In Theorem \ref{AN vanishing}, all the assumptions on $X$, except having only isolated singularities, are satisfied if $X$ is a projective globally $F$-regular 3-fold by Remark \ref{GFR are CM}. In other words, the Akizuki-Nakano vanishing holds for very ample line bundles on projective globally $F$-regular 3-folds with only isolated singularities. 
\end{rem}

Next, we discuss the Akizuki-Nakano vanishing theorem when the singular locus of the variety may have positive dimension. 
Hamm proved that the following form of the Akizuki-Nakano vanishing theorem holds on locally complete intersection projective varieties in characteristic zero. 

\begin{thm}[\textup{\cite[Theorem 0.1]{Hamm}}]\label{Hamm}
Let $X$ be a locally complete intersection complex projective variety and $\mathcal{H}$ be an ample line bundle on $X$.
Let $d$ denote the dimension of $X$ and $s$ denote the dimension of the singular locus of $X$.
Then for all nonnegative integers $i,j$ with $i+j<d-s$, one has
\[H^i(X, \Omega_X^{j} \otimes \mathcal{H}^{-1})=0.\]
\end{thm}

We will prove a characteristic $p$ analog of Theorem \ref{Hamm} in dimension 3. 
We start with the following auxiliary lemma. 

\begin{lem}\label{Bogomolov isolated}
Let $Y$ be a projective 3-fold over a perfect field of characteristic $p>0$ and $\mathcal{L}$ be a line bundle on $Y$.
Suppose that $Y$ has only isolated Cohen-Macaulay singularities and $H^1(Y, \mathcal{L}^{\otimes (-n)})=0$ for all $n \ge 1$.
Then we have 
\[H^0(X, \Omega_Y^{[1]} \otimes \mathcal{L}^{-1})=0.\]
\end{lem}

\begin{proof}
The proof is similar to that of \cite[Theorem 4.5]{Kaw}.
\end{proof}

\begin{prop}\label{Bogomolov}
Let $X$ be a normal projective 3-fold over a perfect field of characteristic $p \ge 5$ and $\mathcal{H}$ be an ample line bundle on $X$.
Suppose that there exists an effective $\Q$-Weil divisor $\Delta$ such that $K_X+\Delta$ is $\Q$-Cartier and $(X, \Delta)$ is strongly $F$-regular and globally $F$-split.
Then we have 
\[H^0(X, (\Omega_X)_{\mathrm{t.f.}} \otimes \mathcal{H}^{-1})=0,\]
where $(\Omega_X)_{\mathrm{t.f.}}$ is the torsion free part of the cotangent sheaf $\Omega_X$.
\end{prop}

\begin{proof}
First note by Remark \ref{hierarchy} that the pair $(X, \Delta)$ is klt.
By \cite[Theorem 1.7]{Bir} and \cite{HaWi}, we can take a terminal model $f: Y \to X$ of the klt pair $(X, \Delta)$, that is, $f$ is a projective birational morphism from a normal projective variety $Y$ such that $\Delta_Y : = f^*(K_X+ \Delta)-K_Y$ is effective and $(Y, \Delta_Y)$ is terminal.
It then follows from Lemma \ref{crepant}, Lemma \ref{crepant SFR} and Remark \ref{hierarchy} that $Y$ is globally $F$-split and has only Cohen-Macaulay isolated singularities.

\begin{cl}
$R^if_* \sO_Y=0$ for all $i>0$. 
\end{cl}
\begin{proof}[Proof of Claim]
It suffices to show that $H^i(V, \sO_V)=0$, where $V:=f^{-1}(U)$, for every open affine subscheme $U \subseteq X$. 
Since $(U, \Delta|_U)$ is globally $F$-regular and $\Delta_Y$ is effective, Lemma \ref{crepant} tells us that $V $ is also globally $F$-regular.
Noting that $V$ is projective over $U$, we have an ample effective Cartier divisor $D$ on $V$ such that $H^i(V, \sO_V(D))=0$.
Combining this with the fact that $H^i(V,\sO_V)$ is a direct summand of $H^i(V, F^e_*\sO_V(D))$ for some $e>0$, we have the vanishing $H^i(V,\sO_V)=0$ as desired. 
\end{proof}

By the above claim and Proposition \ref{kodaira} (2),  
\[
H^1(Y, f^*\mathcal{H}^{\otimes (-n)}) \cong H^1(X, \mathcal{H}^{\otimes(-n)})=0
\]
 for all $n \ge 1$.
It, therefore, follows from Lemma \ref{Bogomolov isolated} that 
\[H^0(X, f_*\Omega_Y^{[1]} \otimes \mathcal{H}^{-1})=H^0(Y, \Omega_Y^{[1]} \otimes f^*\mathcal{H}^{-1})=0.\]
Since $f_* \Omega_Y^{[1]}$ is torsion free and there exist natural morphisms
\[
\Omega_X \to f_* \Omega_Y \to f_* \Omega_Y^{[1]},
\]
which are generically isomorphic, we obtain an inclusion $(\Omega_X)_{\mathrm{t.f.}} \hookrightarrow f_* \Omega_Y^{[1]}$.
Then
\[
H^0(X, (\Omega_X)_{\mathrm{t.f.}} \otimes \mathcal{H}^{-1}) \subseteq H^0(X, f_*\Omega_Y^{[1]} \otimes \mathcal{H}^{-1})=0,\]
which completes the proof.
\end{proof}

In order to prove a vanishing theorem involving cotangent bundles, we use the following lemma on their depths.

\begin{lem}[\textup{\cite[Theorem 1.1]{Hamm}}]\label{depth}
Let $(R, \m)$ be a complete intersection local ring essentially of finite type over a perfect field $k$.
Let $n$ denote the dimension of $R$ and $\delta$ denote the dimension of the singular locus of $\Spec R$.
Then for every integer $0 \le i \le n-\delta$, we have $\mathrm{depth}_{\m} (\Omega_{R/k}^i) \ge n-i$. 
\end{lem}

\begin{thm}\label{AN vanishing non-isolated}
Let $X$ be a locally complete intersection projective 3-fold over a perfect field of characteristic $p \ge 5$ and $\mathcal{H}$ be an ample line bundle on $X$.
Suppose that $X$ is strongly $F$-regular (and in particular normal) and globally $F$-split.
Then for all nonnegative integers $i,j$ with $i+j<2$, one has 
\[H^i(X, \Omega_X^{j} \otimes \mathcal{H}^{-1})=0.\]
\end{thm}

\begin{proof}
The case where $j=0$ is immediate from Proposition \ref{kodaira}.
Since $X$ is a locally complete intersection, $\Omega_X$ is torsion free by \cite[Proposition 8.1]{Lip} (or by Lemma \ref{depth}).
Therefore, the case where $(i,j)=(0,1)$ follows from Proposition \ref{Bogomolov}.
\end{proof}

The following is a characteristic $p$ analog of Theorem \ref{Hamm} for locally complete intersection globally $F$-regular 3-folds. 
\begin{cor}\label{GFR case2}
Let $X$ be a locally complete intersection globally $F$-regular projective 3-fold over a perfect field of characteristic $p>0$. 
Let $\mathcal{H}$ be a globally generated ample line bundle on $X$ such that the morphism $\Phi_{\mathcal{H}}:X \to \mathbb{P}^{N}=\mathbb{P}(H^0(X,\mathcal{H})^*)$ induced by $\mathcal{H}$ is generically \'etale on its image. 
Let $s$ denote the dimension of the singular locus of $X$, and suppose either $s=0$ or $p \ge 5$.
Then for all nonnegative integers $i,j$ with $i+j<3-s$, one has 
\[H^i(X, \Omega_X^{j} \otimes \mathcal{H}^{-1})=0.\]
\end{cor}

\begin{proof}
We first consider the case where $s=0$.
In this case, $H^i(X, \Omega_X^{[j]}\otimes \mathcal{H}^{-1})=0$ for every $i,j \ge 0$ with $i+j <3$ by Remark \ref{GFR case1}.
On the other hand, it follows from Lemma \ref{depth} that $\Omega^j_X$ is reflexive if $j=0,1$ and is torsion free if $j=2$.
Therefore, we have $\Omega^j_X \cong \Omega^{[j]}_X$ for $j=0,1$ and $\Omega^j_X \subseteq \Omega^{[j]}_X$ for $j=2$, which shows $H^i(X, \Omega_X^{j} \otimes \mathcal{H}^{-1})=0$ for every $i,j \ge 0$ with $i+j <3$.

The case where $s=1$ immediately follows from Theorem \ref{AN vanishing non-isolated}.
\end{proof}

We conclude this section with Lefschetz hyperplane type results for globally $F$-split 3-folds deduced from Corollary \ref{GFR case2}.  
\begin{thm}\label{F-split Lefschetz} 
Let $X$ be a globally $F$-split projective 3-fold over a perfect field of characteristic $p>0$ and $H$ be a smooth globally $F$-split very ample Cartier divisor on $X$. 
Suppose that $H^0(X, \omega_X)=0$ and $X$ has only isolated complete intersection singularities. 
Then for all nonnegative integers $i, j$, the natural map
\[
H^i(X, \Omega_X^j) \to H^i(H, \Omega_H^j)
\]
is an isomorphism if $i+j<2$ and is injective if $i+j<3$. 
\end{thm}
\begin{proof}
First note that $\mathscr{T}\mathrm{or}_1(\Omega_X^j, \sO_H)=0$ for all $2 \ge j \ge 0$, because $H$ is a Cartier divisor and $\Omega_X^j$ is torsion free by Lemma \ref{depth}. 
Therefore, 
\[0 \to \Omega_X^j(-H) \to \Omega_X^j \to \Omega_X^j|_H \to 0\]
is exact for all $2 \ge j \ge 0$. 
By Theorem \ref{AN vanishing} or an argument similar to the proof of Corollary \ref{GFR case2},  
we have $H^i(X, \Omega_X^j(-H))=0$ for all nonnegative integers $i,j$ with $i+j<3$. 
Thus, the map $H^i(X, \Omega_X^j) \to H^i(X, \Omega_X^j|_H)$ is an isomorphism if $i+j<2$ and is injective if $i+j<3$. 

Again, since $H$ is a smooth Cartier divisor on $X$, we have the following exact sequence of locally free sheaves:
\[0 \to \sO_X(-H)|_H \to \Omega_X|_H \to \Omega_H \to 0.\]
For each integer $j \ge 1$, by \cite[II. Exc. 5.16 (d)]{Hart}, it induces the exact sequence 
\[0 \to \Omega_H^{j-1}(-H) \to \Omega_X^j|_H \to \Omega_H^j \to 0.\]
Applying the Akizuki-Nakano vanishing theorem for smooth globally $F$-split surfaces (see the first paragraph of \S \ref{vanishing section}), 
we have $H^i(H, \Omega_H^{j-1}(-H))=0$ for all nonnegative integers $i, j$ with $i+j<3$. 
Together with Proposition \ref{kodaira}, this shows that the map $H^i(X, \Omega_X^j|_H) \to H^i(H, \Omega_H^j)$ is an isomorphism if $i+j<2$ and is injective if $i+j<3$.  
Compositing this map with the above map, we obtain the assertion. 
\end{proof}

We say that a line bundle $\mathcal{L}$ on a variety $X$ is \textit{normally generated} if the multiplication map $H^0(X, \mathcal{L})^{\otimes n} \to H^0(X, \mathcal{L}^{\otimes n})$ is surjective for every integer $n \ge 1$. 
Normally generated ample line bundles are very ample. 
\begin{cor}\label{Fano Lefschetz}
Let $X$ be a globally $F$-split projective Fano 3-fold over an infinite perfect field of characteristic $p>0$ which has only isolated complete intersection singularities. 
Suppose that $-K_X$ is normally generated and $H$ is a general member of the anti-canonical system $|-K_X|$. 
Then for all nonnegative integers $i, j$, the natural map
\[
H^i(X, \Omega_X^j) \to H^i(H, \Omega_H^j)
\]
is an isomorphism if $i+j<2$ and is injective if $i+j<3$. 
\end{cor}
\begin{proof}
Applying an unpublished result of N.~Hara or \cite[Theorem 5.8]{DN} to the anti-canonical ring $\bigoplus_{n \ge 0}H^0(X, \sO_X(-nK_X))$, we see that $H$ is globally $F$-split. 
Since $H$ is a general member of the very ample linear system $|-K_X|$ and $X$ has only isolated singularities, $H$ is smooth by Bertini's theorem. 
The assertion therefore follows from Theorem \ref{F-split Lefschetz}. 
\end{proof}

\section{Deformations of globally \texorpdfstring{$F$-split}{F-split} Fano 3-folds}
In this section, we study the deformations of globally $F$-split Fano 3-folds, making use of the results in \S\ref{vanishing section}. 
First we recall some basic terminology from the theory of deformations. 
\begin{defn}[cf.~\cite{NL}, \cite{Ser}, \cite{Vi}]
Let $X$ be an algebraic scheme over an algebraically closed field $k$.
\begin{enumerate}
\item[\textup{(i)}] 
Let $T$ be an algebraic scheme over $k$ and $t \in T$ be a closed point.
A \textit{deformation} of $X$ over $T$ with reference point $t$ is a pair $(\mathcal{X}/T, i)$ of a flat morphism $\mathcal{X} \to T$ of $k$-schemes and an isomorphism $i: X \xrightarrow{\ \sim\ } \mathcal{X} \times_T \Spec k(t)$ over $\Spec k$.
\item[\textup{(ii)}] 
Two deformations $(\mathcal{X},i)$ and $(\mathcal{X}', i')$ of $X$ over $T$ with the same reference point $t \in T$ are said to be \textit{equivalent} if there exists an isomorphism $\phi : \mathcal{X} \xrightarrow{\ \sim \ } \mathcal{X}'$ over $T$ such that the induced isomorphism $\phi_t : \mathcal{X} \times_T \Spec k(t) \to \mathcal{X}' \times_T \Spec k(t)$ satisfies the following commutative diagram.
\[
\xymatrix{
& \ar_-{i}[ld] X \ar^-{i'}[rd] &\\
\mathcal{X} \times_T \Spec k(t) \ar^-{\phi_t}[rr] && \mathcal{X}' \times_T \Spec k(t) 
}
\]
\item [\textup{(iii)}]
Let $(A, \m_A)$ be an Artinian local $k$-algebra with residue field $k$.
We define $\mathrm{Def}_X(A)$ as the set of equivalent classes of deformations of $X$ over $\Spec A$ with reference point $\m_A \in \Spec A$.
For a local homomorphism $f: A \to A'$ of Artinian local $k$-algebras with residue field $k$, we define
\[
\mathrm{Def}_X(f) : \mathrm{Def}_X(A) \to \mathrm{Def}_X(A')
\]
as the map that sends an equivalent class of a deformation $(\mathcal{X}, i)$ of $X$ over $\Spec A$ to the equivalent class of the deformation $(\mathcal{X}' : = \mathrm{X} \times_{\Spec A} \Spec{A'}, i')$, where $i'$ is the composition 
\[
i': X \xrightarrow{\ i \ } \mathcal{X} \times_{\Spec A} k \xrightarrow{\ \sim\ } \mathcal{X}' \times_{\Spec A'} k.
\]
\item [\textup{(iv)}]
We say that the deformations of $X$ are \textit{unobstructed} if for any Artinian local $k$-algebras $(A, \m)$ and $(A', \m')$ with residue field $k$ and for any surjective ring homomorphism $f : A \to A'$ with $\m \cdot \mathrm{Ker}\; f=0$, the map $\mathrm{Def}_X(f) : \mathrm{Def}_X(A) \to \mathrm{Def}_X(A')$ is surjective.
\item [\textup{(v)}]
We say that $X$ admits a \textit{smoothing} if there exist a smooth (not necessarily projective) curve $C$ over $k$, a closed point $c \in C$ and a deformation $\pi: \mathcal{X} \to C$ of $X$ over $C$ with reference point $c$ such that $\pi$ is smooth over $C \setminus \{c\}$. 
\end{enumerate}

\end{defn}

Namikawa \cite{Na} carefully examined the deformations of complex Fano 3-folds with Gorenstein terminal singularities. 

\begin{thm}[\textup{\cite[Proposition 3]{Na}}]\label{namikawa result1}
Let $X$ be a complex projective Fano 3-fold with Gorenstein terminal singularities. 
Then the deformations of $X$ are unobstructed. 
\end{thm}

\begin{rem}
Theorem \ref{namikawa result1} was recently generalized by Sano \cite{Sa} to the case of complex projective weak Fano 3-folds with terminal singularities. 
\end{rem}

We first prove a characteristic $p$ analog of Theorem \ref{namikawa result1}. 
A variety $X$ over a perfect field $k$ of characteristic $p>0$ is said to be \textit{liftable} to the Witt ring $W(k)$ of $k$ if there exists a flat morphism $\mathcal{X} \to \Spec W(k)$ of schemes with closed fiber isomorphic to $X$. 

\begin{thm}\label{unobstructed}
Let $X$ be a globally $F$-split projective Fano 3-fold over an algebraically closed field of characteristic $p>0$ which has only isolated complete intersection singularities.  
Suppose that there exists an integer $1 \le i_0 \le p$ such that $|-i_0K_X|$ is base point free and the morphism $\varphi_{|-i_0K_X|}:X \to \mathbb{P}^N=\mathbb{P}(H^0(X,\sO_X(-i_0K_X))^*)$ induced by $|-i_0K_X|$ is generically \'etale on its image. 
Then the deformations of $X$ are unobstructed and, in particular, $X$ is liftable to $W(k)$.
\end{thm}
\begin{proof}
First note that $H^0(X, \omega_X)=0$, because $X$ is Fano.  
The anti-canonical line bundle $\omega_X^{-1}$ satisfies the conditions (i)-(iii) in Proposition \ref{vanishing (1,1)}. 
Indeed, (i) is verified by Lemma \ref{reflexive serre vanishing} and (ii), (iii) are done by Proposition \ref{kodaira} and the vanishing of $H^0(X, \omega_X)$. 
Since $\Omega_X$ is reflexive by \cite[Proposition 8.1]{Lip} (or by Lemma \ref{depth}), 
it follows from Proposition \ref{vanishing (1,1)} that $H^1(X, \Omega_X \otimes \omega_X)=0$. 
We see by Serre duality that $\mathrm{Ext}^2(\Omega_X, \mathcal{O}_X)=0$, which implies the first assertion by \cite[Proposition 6.4 (c)]{Vi}.
Moreover, it follows from \cite[Proposition 6.3]{Vi} and \cite[Theorem 22.3]{Mat} that $X$ is liftable to $W(k)$.
\end{proof}

\begin{rem}
Fujita's very ampleness conjecture says that if $X$ is an $n$-dimensional projective variety over an algebraically closed field and $\mathcal{L}$ is an ample line bundle on $X$, then $\omega_X \otimes \mathcal{L}^{\otimes (n+2)}$ is very ample. 
If Fujita's conjecture holds for $\omega_X$, then the integer $i_0$ in Theorem \ref{unobstructed} exists when $p \ge 5$. 
In particular, it exists if $|-K_X|$ is base point free and $p \ge 5$ (see \cite[Example 1.8.23]{La}). 
\end{rem}

Next, we discuss when a globally $F$-split Fano 3-fold admits a smoothing. 
We use the following vanishing theorem, which can be viewed as a characteristic $p$ analog of the small resolution case of \cite[Theorem 2 (a')]{St}. 
\begin{prop}\label{relative vanishing}
Let $X$ be a normal globally $F$-split projective 3-fold over a perfect field of characteristic $p>0$ with $H^0(X, \omega_X)=0$. 
Let $\mathcal{L}$ be an ample line bundle on $X$, and suppose there exists an integer $1 \le i_0 \le p$ such that $\mathcal{L}^{\otimes i_0}$ is globally generated and the morphism $\Phi_{\mathcal{L}^{\otimes i_0}}:X \to \mathbb{P}^{N}=\mathbb{P}(H^0(X,\mathcal{L}^{\otimes i_0})^*)$ induced by $\mathcal{L}^{\otimes i_0}$ is generically \'etale on its image.
Suppose that $f:G \to X$ is a small resolution, that is, $G$ is a smooth integral algebraic space over $k$ and $f$ is a proper small birational morphism.
If $R^1 f_* \sO_{G}=0$, then 
\[H^1(G, \Omega_G \otimes f^*\mathcal{L}^{-1})=0.\] 
\end{prop}
\begin{proof}
First note that $G$ is globally $F$-split by Lemma \ref{space small GFS}.
Let $Z \subseteq \PP^N_k$ be the scheme theoretic image of $\Phi_{\mathcal{L}^{\otimes i_0}}$ and $\Phi := \Phi_{\mathcal{L}^{\otimes i_0}} \circ f : G \to Z$ be the composite map of $\Phi_{\mathcal{L}^{\otimes i_0}}$ and $f$. 
Since $\Phi$ is generically \'{e}tale and $\Phi^* (\sO_{\PP^N}(1)|_Z) \cong f^* \mathcal{L}^{\otimes i_0}$, it suffices to verify that the line bundle $f^*\mathcal{L}$ satisfies the conditions (i)-(iii) in Proposition \ref{vanishing on stack}.

By assumption and Lemma \ref{space small}, we have $R^j f_* \sO_{G}=0$ for all $j \ge 1$ and $f_* \sO_G= \sO_X$.
Therefore, $H^2(G, f^*\mathcal{L}^{\otimes (\pm p^e)}) \cong H^2(X, \mathcal{L}^{\otimes (\pm p^e)})=0$ for all integers $e \ge 1$ by Proposition \ref{kodaira}, that is, the condition (ii) is satisfied for $f^*\mathcal{L}$. 
Similarly, $H^3(G, f^*\mathcal{L}^{\otimes (p-i_0)p^e}) \cong H^3(X, \mathcal{L}^{\otimes (p-i_0)p^e})=0$ by the assumption $H^0(X, \omega_X)=0$ when $i_0=p$ and by Proposition \ref{kodaira} when $i_0<p$. 
In other words, the condition (iii) is also satisfied for $f^*\mathcal{L}$. 

Next, we will check the condition (i). 
The Leray spectral sequence 
\[E^{i,j}_2=H^i(X, R^jf_*(T_G \otimes \omega_G) \otimes \mathcal{L}^{\otimes p^e}) \Rightarrow E^{i+j}=H^{i+j}(G, T_G \otimes \omega_G \otimes f^*\mathcal{L}^{\otimes p^e})\]
induces an isomorphism $E^{0,2}_2 \cong E^2$ for all sufficiently large $e$, because if $i>0$, then $E^{i,j}_2=0$ for all sufficiently large $e$ by Serre vanishing. 
On the other hand, since $f$ is small, $R^2f_*(T_G \otimes \omega_G)=0$ by Lemma \ref{space small} and therefore $E^{i,2}_2=0$ for all $i$ and $e$. 
In conclusion, for all sufficiently large $e$, one has 
\[
0=H^0(X, R^2 f_*(T_G \otimes \omega_G) \otimes \mathcal{L}^{\otimes p^e}) \cong 
H^2(G, T_G \otimes \omega_G \otimes f^*\mathcal{L}^{\otimes p^e}), 
\]
which is the Serre dual of $H^1(G, \Omega_G \otimes f^*\mathcal{L}^{\otimes (-p^e)})$ (see \cite[Theorem 2.22]{Ni} for Serre duality for algebraic spaces).  
Thus, (i) is satisfied for $f^*\mathcal{L}$. 
\end{proof}

\begin{defn}\label{def A1 sing}
Let $x \in X$ be a closed point of a three-dimensional algebraic variety $X$ over an algebraically closed field $k$, and $\widehat{\sO_{X,x}}$ denotes the $\m_{X,x}$-adic completion of the local ring $\sO_{X,x}$. 
We say that $x \in X$ is an \textit{ordinary double point}\footnote{such a singularity is also called an $A_1$-singularity.} if the completion $\widehat{\sO_{X,x}}$ is isomorphic to $k[[s,t,u,v]]/(st-uv)$, where $k[[s,t,u,v]]$ is the 4-dimensional formal power series ring over $k$. 
\end{defn}

The following lemma is well-known to experts, but we provide a proof for the convenience of the reader. 
\begin{lem}[cf.~\cite{Fr}]\label{basic resolution}
Let $V= \Spec k[s,t,u,v]/(st-uv)$, where $k[s,t,u,v]$ is the 4-dimensional polynomial ring over an algebraically closed field $k$, and $Z \subseteq V$ be the closed subscheme defined by the ideal sheaf $(s,u) \subseteq \sO_V$ and $\pi: W \to V$ be the blow-up along $Z$. 
Then $\pi$ is a small resolution such that $R^1 \pi_* T_W= R^1 \pi_* \sO_W=0$.
\end{lem}

\begin{proof}
Let $v \in V$ denote the singular point of $V$, and set $C:=\pi^{-1}(v)$. 
Since $Z$ is Cartier outside $v$, the morphism $\pi$ is an isomorphism over $U : = V \setminus \{ v \}$.
One can easily see that $W$ is regular, $ C \cong \PP^1_k$ and $I/I^2 \cong \sO_{\PP^1}(1)^{\oplus 2}$, where $I \subseteq \sO_W$ is the defining ideal of $C$. 
In particular, $\pi$ is a small resolution. 
On the other hand, $V$ is strongly $F$-regular by Fedder's criterion (\cite[Theorem 2.3]{Gla}). 
It, therefore, follows along the same lines as the proof of the claim in the proof of Theorem \ref{Bogomolov} that $R^1 \pi_* \sO_W=0$. 

It remains to show that $R^1\pi_* T_{W}=0$.
Since $\pi$ is an isomorphism over $U$ and the $\m_v$-adic completion $\sO_{V,v} \to \widehat{\sO_{V, v}}$ is faithfully flat, it is enough to show that the $\m_v$-adic completion of the stalk of $R^1 f_* T_{W}$ at $v$ is zero.
By the theorem of formal functions, it suffices to prove that $H^1(W, (\sO_{W}/I^n) \otimes_{W} T_{W})=0$ for every $n \ge 1$. 
Looking at the long exact sequence of cohomology associated to 
\[
0 \to I^n/I^{n+1} \to \sO_{W}/I^{n+1} \to \sO_{W}/I^n  \to 0
\]
tensored with $T_W$, we can reduce the problem to showing that for all $n \ge 0$, 
\[H^1(C, (I^n/I^{n+1}) \otimes_C T_W|_C)=0.\] 
It follows from \cite[II. Theorem 8.21.A]{Hart} that for every $n \ge 1$, we have the isomorphisms 
\[I^n/I^{n+1} \cong S^n(I/I^2) \cong {\sO_{\PP^1}(n)}^{\oplus (n+1)}.\]
Combining this with the dual of the second exact sequence
\[
0 \to T_C \to T_{W}|_C \to (I/I^2)^* \to 0,
\] 
we can verify the vanishing $H^1(C, (I^n/I^{n+1}) \otimes_C T_W|_C)=0$, which completes the proof.
\end{proof}

\begin{prop}\label{prop on a1}
Let $X$ be a three-dimensional variety over an algebraically closed field.
Suppose that $X$ has only ordinary double points.
There then exists a small resolution $f:G \to X$ from a smooth integral algebraic space $G$ that satisfies $R^1f_*T_G= R^1f_*\sO_G=0$.  
\end{prop}

\begin{proof}
Let $\{x_1, \dots, x_n \} \subseteq X$ be the set of all singular points of $X$, and set $U := X \setminus \{x_1, \dots, x_n \}$.
Let $W,V,Z, \pi$ be as in Lemma \ref{basic resolution} and $v \in V$ be the singular point of $V$.
Since each $x_i \in X$ is an ordinary double point, it follows from \cite[Corollary 2.6]{Art} that for each $i$, there exist a variety $V_i$ whose singular locus is a single point $v_i \in V$ and \'{e}tale morphisms $\phi_i: V_i \to X$ and $\psi_i: V_i \to V$ such that $\phi_i(v_i)=x_i$ and $\psi_i(v_i)=v$.
\[
\xymatrix{
& V_i \ar_-{\phi_i}[ld] \ar^-{\psi_i}[rd] & \\
X & & V
}
\]

Let $\pi_i : W_i \to V_i$ be the blow-up along the closed subscheme $Z_i:=\psi_i ^{-1}(Z) \subseteq V_i$.
Then $\pi_i$ is a small resolution with $R^1(\pi_i)_* \sO_{W_i}=R^1(\pi_i)_* T_{W_i}=0$,  because the diagram 
\[
\xymatrix{
W_i \ar^-{}[d] \ar^-{\pi_i}[r] & V_i  \ar^-{\psi_i}[d]\\
W \ar^-{\pi}[r]& V 
}
\]
is Cartesian. 
Set $V_i^{\circ} : =\phi^{-1}_i(U) \subseteq V_i$ and $W_i^{\circ} : = \pi_i^{-1}(V_i^{\circ}) \subseteq W_i$.
Since $Z_i|_{V_i^{\circ}}$ is a Cartier divisor, $\pi_i$ induces the isomorphism $W_i^{\circ} \cong V_i^{\circ}$.
We define the open subscheme $H_i$ of $X$ by $H_i:= X \setminus \{x_{i+1}, x_{i+2}, \dots, x_n\}$, and let $\iota_i : H_{i-1} \hookrightarrow H_i$ denote the natural inclusion. 
Note that $\varphi_i:V_i \to X$ factors through $H_i$. 
For each $0 \le i \le n$, we will construct a small resolution $f_i: G_i \to H_i$ that is an isomorphism over $U \subseteq H_i$ and satisfies $R^1 (f_i)_*\sO_{G_i}=R^1 (f_i)_*\sO_{H_i}=0$.

If $i=0$, then we define $G_0 : =H_0=U$ and $f_0 := \mathrm{id}: G_0 \to U$.
Suppose that $i \ge 1$ and $f_{i-1}: G_{i-1} \to H_{i-1}$ is a small resolution as above.
It follows from \cite[Lemma 0DVJ]{Sta} that there exist an algebraic space $G_i$, an open immersion $a_i : G_{i-1} \hookrightarrow G_i$ and an \'{e}tale morphism $b_i: W_i \to G_i$ such that $(G_{i-1} \subset G_i,b_i:W_i \to G_i)$ is an elementary distinguished square and the diagram 
\[
\xymatrix{
W_i^{\circ} \ar@{^{(}->}[r] \ar^-{}[d] & W_i \ar^-{b_i}[d]\\
G_{i-1} \ar^-{a_i}[r] & G_i
}
\]
is a pushout in the category of algebraic space over $k$. 
The universality of pushouts yields a morphism $f_i : G_i \to H_i$ such that each square in the diagram 
\[
\xymatrix{
& W_i^{\circ} \ar[rr] \ar[dd]|(.5)\hole \ar[ld] && W_i \ar^-{b_i}[dd] \ar^-{\pi_i}[ld]\\
V_i^{\circ} \ar[rr] \ar[dd]  && V_i \ar^(.4){\phi_i}[dd]  & \\
& G_{i-1} \ar_(.4){a_i}[rr]|(.52)\hole \ar^(0.35){f_{i-1}}[ld]&& G_i \ar^-{f_i}[ld]\\
H_{i-1} \ar^-{\iota_i}[rr] && H_i &
}
\]
is commutative. Then by Proposition \ref{cube}, the following diagram is Cartesian.
\[
\xymatrix{
G_{i-1} \sqcup W_i \ar^-{a_i \sqcup b_i}[rr] \ar^-{f_{i-1} \sqcup \pi_i}[dd] && G_i \ar^-{f_i}[dd]\\
\\
H_{i-1} \sqcup V_i \ar^-{\iota_i \sqcup \phi_i}[rr] && H_i
}
\]
Since the horizontal arrows are \'{e}tale surjective and $f_{i-1} \sqcup \pi$ is proper, the morphism $f_i$ is also proper.
We note that this implies that $G_i$ is separated.
Moreover, it follows from \cite[Lemma 073K]{Sta} that $R^1 (f_i)_*\sO_{G_i}=R^1(f_i)_* T_{G_i}=0$.
It also follows from the same diagram that $f_i$ is a small birational that is an isomorphism over $U$, as desired.

We now obtain the assertion of the proposition by choosing $G:=G_n$ and $f:=f_n: G_n \to H_n=X$.
\end{proof}

A complex Fano 3-fold with only ordinary double points is smoothable by a flat deformation. 
\begin{thm}[\textup{\cite[Corollary 4.2]{Fr}, \cite[Proposition 4]{Na}}]\label{namikawa result2}
Let $X$ be a complex projective Fano 3-fold with ordinary double points. Then $H^2(X, T_X)=0$ and, in particular, $X$ admits a smoothing. 
\end{thm}

\begin{rem}
Namikawa \cite{Na} in fact proved that every complex projective Fano 3-fold $X$ with Gorenstein terminal singularities admits a smoothing even if $H^2(X, T_X) \ne 0$. 
\end{rem}

We prove a characteristic $p$ analog of Theorem \ref{namikawa result2}. 

\begin{thm}\label{smoothing}
Let $X$ be a globally $F$-split projective Fano 3-fold over an algebraically closed field $k$ of characteristic $p>0$ 
which has only ordinary double points.  
Suppose that there exists some $1 \le i_0 \le p$ such that $|-i_0K_X|$ is base point free and the morphism $\varphi_{|-i_0K_X|}:X \to \mathbb{P}^N=\mathbb{P}(H^0(X,\sO_X(-i_0K_X))^*)$ induced by $|-i_0K_X|$ is generically \'etale on its image. 
Then $H^2(X, T_X)=0$ and, in particular, $X$ admits a smoothing.
\end{thm}
\begin{proof}
By Proposition \ref{prop on a1}, there exists a small resolution $f:G \to X$ such that $R^1f_*\sO_G=R^1f_*T_G=0$. 
Substituting $\mathcal{L}=\omega_X^{-1}$ in Proposition \ref{relative vanishing}, one has 
\[
0=H^1(G, \Omega_G \otimes f^*\omega_X)=H^1(G, \Omega_G \otimes \omega_G),
\]
which is the Serre dual of $H^2(G, T_G)$. 
Since $f$ is small, $f_*T_G \cong T_X$ and $R^jf_*T_G=0$ for all $j \ge 2$ by Lemma \ref{space small}.  
From this, together with the vanishing $R^1f_*T_G=0$, we obtain the isomorphisms
\[H^2(X, T_X) \cong H^2(X, f_*T_G) \cong H^2(G, T_G)=0.\]
It then follows from \cite[Theorem 4.14 and Lemma 5.1]{NL} that $X$ is smoothable by a flat deformation. 
\end{proof}

\begin{rem}
Since the modulo $p$ reduction of a complex projective Fano 3-fold with only Gorenstein terminal singularities is a globally $F$-split 3-fold with only isolated complete intersection singularities for sufficiently large primes $p$ (see \cite[Theorem 1.2]{SS}), Theorem \ref{unobstructed} (resp. Theorem \ref{smoothing}) gives a purely algebraic proof of Theorem \ref{namikawa result1} (resp. Theorem \ref{namikawa result2}). 
\end{rem}

\section{Kodaira vanishing for thickenings}\label{thickening section}
Throughout this section, let $X$ be a locally complete intersection closed subvariety of the projective space $\PP^N_k$ over a perfect field $k$.
Let $d$ denote the dimension of $X$ and $s$ denote the dimension of the singular locus of  $X$. 
For each integer $t \ge 1$, let $X_t \subseteq \PP^N_k$ be the $t$-th thickening of $X$, that is, the closed subscheme defined by $\mathcal{I}^t$, where $\mathcal{I} \subseteq \sO_{\PP^N_k}$ is the defining ideal sheaf of $X$ in $\PP^N_k$.

Bhatt-Blickle-Lyubeznik-Singh-Zhang \cite{BBLSZ} proved that the Kodaira vanishing theorem holds on  $X_t$ in characteristic zero.  
\begin{thm}[\textup{\cite[Theorem 3.1]{BBLSZ}}]\label{thicknings in char. 0}
Suppose that $k$ is of characteristic zero. 
Then for every integers $i<d-s$, $\ell \ge 1$ and $t \ge 1$, we have 
\[H^i(X_t, \sO_{X_t}(-\ell))=0.\] 
\end{thm}

We pursue a characteristic $p$ analog of Theorem \ref{thicknings in char. 0}, making use of our vanishing results in \S \ref{vanishing section}. 

\begin{prop}\label{vanishing D}
Suppose that $H^i(X, \Omega^{j}_X(-\ell))=0$ for every integers $\ell \ge 1$ and $i, j \ge 0$ with $i+j <d-s$.
Then for every integers $i < d-s$, $\ell \ge 1$ and $t \ge 1$, we have 
\[H^i(X, D_t(\mathcal{I}/\mathcal{I}^2)(-\ell))=0.\]
\end{prop}

\begin{proof}
Fix integers $t, \ell \ge 1$.
Since $X$ is a locally complete intersection, the sequence
\[
0 \to \mathcal{I}/\mathcal{I}^2 \to \Omega_{\PP^N}|_X \to \Omega_X \to 0
\]
is exact also on the left.
Applying Proposition \ref{divided cpx}, we obtain the complex
\[
0 \to D_t(\mathcal{I}/\mathcal{I}^2) \xrightarrow{\ d^{-1} \ } C^0 \xrightarrow{\ d^0\ } C^1 \xrightarrow{\ d^1\ } C^2 \xrightarrow{\ d^2\ } \cdots ,
\]
where $C^i : = D_{t-i}(\Omega_{\PP^N}|_X) \otimes \Omega_X^{i}$.
Since $\Omega_X$ is locally free on the smooth locus $U \subseteq X$, the complex is exact on $U$.

For every $i\ge 0$, we define the subsheaves $B^i, Z^i$ of $C^i$ by $B^i:=\mathrm{Im}\; d^{i-1}$ and $Z^i:=\mathrm{Ker}\; d^i$, respectively.
We note that $B^i|_U = Z^i|_U$ and they are locally free $\sO_U$-modules for every $i$.
Indeed, since $C^i|_U$ is locally free, it follows from the exact sequence
\[ 0 \to B^i|_U \to C^i|_U \to B^{i+1}|_U \to 0 \]
that $B^{i}|_U$ is locally free if so is $B^{i+1}|_U$.
Combining with the fact that $C^i|_U=0$ for all sufficiently large $i$, we conclude that $B^i|_U$ are locally free for all $i$.
We also note that $d^{-1}$ is injective, because it is injective on $U$ and $D_t(\mathcal{I}/\mathcal{I}^2)$ is torsion free.
In particular, we have $B^0 \cong D_t(\mathcal{I}/\mathcal{I}^2)$.

Next, we will verify that $H^i(X, C^j (-\ell))=0$ for every $i, j \ge 0$ with $i+j<d-s$.
Fix such integers $i,j$.
Applying Proposition \ref{divided cpx} to the Euler sequence
\[
0 \to \Omega_{\PP^N}|_X \to \sO_X(-1)^{\oplus N+1} \to \sO_X \to 0
\]
restricted to $X$, we obtain the exact sequence
\[
0 \to D_{t-j}(\Omega_{\PP^N}|_X) \to D_{t-j}(\mathcal{E}) \to D_{t-j-1}(\mathcal{E}) \to 0,
\]
where we write $\mathcal{E}:=\sO_X(-1)^{\oplus N+1}$.
Since $D_{t-j-1}(\mathcal{E})$ is locally free, we have the exact sequence 
\[
H^{i-1}(X, D_{t-j-1}(\mathcal{E}) \otimes \Omega^{j}_X (-\ell)) \to H^i(X,  C^j(-\ell)) \to H^i(X,D_{t-j}(\mathcal{E}) \otimes \Omega^{j}_X (-\ell)).
\]
On the other hand, since $D_{m}(\mathcal{E}) = (S^m(\mathcal{E}^*))^*$ is a direct sum of $\sO_X(-m)$ for every integer $m$, it follows by assumption that
\[
H^{i-1}(X, D_{t-j-1}(\mathcal{E}) \otimes \Omega^{j}_X (-\ell))=0, \quad 
H^{i}(X, D_{t-j}(\mathcal{E}) \otimes \Omega^{j}_X (-\ell))=0.
\]
Thus, we have $H^i(X, C^j (-\ell))=0$.

The assertion in the proposition is clear if $s=d-1$, because 
\[
H^0(X, D_t(\mathcal{I}/\mathcal{I}^2)(-\ell)) \cong H^0(X, B^0 (-\ell)) \subseteq H^0(X, C^0(-\ell)) =0. 
\]
We therefore assume that $s \le d-2$. 
Let $Y : = X \setminus U$ be the singular locus of $X$.
We will verify that $B^i=Z^i$ and $\depth_Y(B^i) : = \min \{ \depth_x(B^i) \mid x \in Y \} \ge d-s-i$ for each $i=0,1, \dots, d-s-2$.
When $i=0$, since $B^0 \cong D_t(\mathcal{I}/\mathcal{I}^2)$ is locally free and $X$ is Cohen-Macaulay, $\depth_Y(B^0) =\mathrm{Codim}(Y, X) = d-s$, and since $B^0$ is reflexive and the inclusion $B^0 \subseteq Z^0$ is the identity on $U$, we have $B^0 =Z^0$. 
When $1 \le i \le d-s-2$, by induction on $i$, we may assume that $\depth_Y(B^{i-1}) \ge d-s-i+1$ and the sequence
\[
0 \to B^{i-1} \to C^{i-1} \to B^i \to 0
\]
is exact. 
It follows from Lemma \ref{depth} that $\depth_Y(C^{i-1}) \ge d-s-i+1$, which implies $\depth_Y(B^i) \ge d-s-i$.
Combining with the fact that $B^i|_U$ is locally free, $B^i$ satisfies the Serre condition $(S_2)$.
Therefore, $B^i$ is reflexive, which proves $B^i=Z^i$.

Fix an integer $0 \le i <d-s$.
The exact sequence 
\[
0 \to B^0 \to C^0 \to B^1 \to 0
\]
induces the exact sequence
\[
H^{i-1}(X, C^0(-\ell)) \to H^{i-1}(X, B^1(-\ell)) \to H^i(X, B^0(-\ell)) \to H^i(X, C^0(-\ell)).
\]
Since $H^{i-1}(X,C^0(-\ell))=H^i(X, C^0(-\ell))=0$, we have 
\[
H^{i}(X, B^0(-\ell)) \cong H^{i-1}(X, B^1(-\ell)).
\]
Similarly, we have
\[
H^i(X, B^0(-\ell)) \cong H^{i-1}(X, B^1(-\ell)) \cong H^{i-2}(X, B^2(-\ell)) \cong \cdots \cong H^0(X, B^i(-\ell)).
\]
Since $B^0 \cong D_t(\mathcal{I}/\mathcal{I}^2)$ and $B^i \subseteq C^i$, we have
\[
H^i(X, D_t(\mathcal{I}/\mathcal{I}^2)(-\ell)) \cong H^0(X, B^i(-\ell)) \subseteq H^0(X, C^i(-\ell))=0,
\]
which completes the proof of the proposition.
\end{proof}

\begin{thm}\label{vanishing thickening}
Let $t \ge 1$ be an integer, and assume the following two conditions hold. 
\begin{enumerate}
\item [$\textup{(i)}$] $H^i(X, \Omega^{j}_X(-\ell))=0$ for every integers $\ell \ge 1$ and $i, j \ge 0$ with $i+j <d-s$.
\item [$\textup{(ii)}$] One of the following conditions is satisfied:
\begin{enumerate}
\item [$\textup{(a)}$] $k$ is of characteristic zero, 
\item [$\textup{(b)}$] $k$ is of characteristic $p \ge t$, 
\item [$\textup{(c)}$] the normal bundle $\mathcal{N}_{X/\PP^N}$ is a direct sum of line bundles.
\end{enumerate}
\end{enumerate}
Then for every $i<d-s$ and $\ell \ge 1$, we have 
\[H^i(X_t, \sO_{X_t}(-\ell))=0.\]
\end{thm}

\begin{proof}
Fix integers $\ell, t \ge 1$ and $0 \le i <d-s$.
If $t=1$, then the assertion follows from Proposition \ref{kodaira} (2).
We consider the case where $t \ge 2$.
By induction on $t$, we may assume that $H^i(X_{t-1}, \sO_{X_{t-1}}(-\ell))=0$.

The exact sequence
\[
0 \to \mathcal{I}^{t-1}/\mathcal{I}^{t} \to \sO_{X_t} \to \sO_{X_{t-1}} \to 0
\]
induces the exact sequence
\[
H^i(X, (\mathcal{I}^{t-1}/\mathcal{I}^t)(-\ell)) \to H^i(X_t, \sO_{X_t}(-\ell)) \to H^i(X_{t-1}, \sO_{X_{t-1}}(-\ell)).
\]
On the other hand, it follows from \cite[II. Theorem.8.21.A(e)]{Hart} and Lemma \ref{div sym} that 
\[
\mathcal{I}^{t-1}/\mathcal{I}^t \cong S^{t-1}(\mathcal{I}/\mathcal{I}^2) \cong D_{t-1}(\mathcal{I}/\mathcal{I}^2).
\]
Therefore, by Proposition \ref{vanishing D}, we have $H^i(X, (\mathcal{I}^{t-1}/\mathcal{I}^t)(-\ell))=0$. 
Combining with the above exact sequence, we have $H^i(X_t, \sO_{X_t}(-\ell))=0$.
\end{proof}

\begin{rem}\label{alternative proof}
As an application of Theorem \ref{vanishing thickening}, together with Theorem \ref{Hamm}, we obtain an alternative proof of Theorem \ref{thicknings in char. 0}. 
\end{rem}

\begin{thm}\label{thicknings in char. p}
Suppose that $t \ge 1$ is an integer, $k$ is of characteristic $p\ge 5$ and $X$ is globally $F$-regular of dimension $d \le 3$. 
Assume in addition that one of the following conditions is satisfied: 
\begin{enumerate}
\item [$\textup{(i)}$] $t$ is less than or equal to $p$, 
\item [$\textup{(ii)}$] the normal bundle $\mathcal{N}_{X/\PP^N}$ is a direct sum of line bundles.
\end{enumerate}
Then for every integers $i < d-s$, $\ell \ge 1$ and $t \ge 1$, we have 
\[H^i(X_t, \sO_{X_t}(-\ell))=0.\]
\end{thm}

\begin{proof}
The assertion follows from Theorem \ref{dim 2}, Corollary \ref{GFR case2} and Theorem \ref{vanishing thickening}.
\end{proof}

\begin{rem}
If $s=0$, then by Remark \ref{GFR case1} and Theorem \ref{vanishing thickening}, Theorem \ref{thicknings in char. p} holds without assuming that $p \ge 5$. 
\end{rem}

\begin{ques}
Does the assertion in Theorem \ref{thicknings in char. p} hold if we drop the assumption that $t$ is less than or equal to $p$ or $\mathcal{N}_{X/\PP^N}$ is a direct sum of line bundles? 
\end{ques}


\end{document}